\documentclass[a4paper, 11pt]{amsart}   
\usepackage{mathptmx, amssymb,amscd,latexsym, eulervm}   
\usepackage{amsmath}
\usepackage{amsthm}
\usepackage{mathdots}
\usepackage{hyperref}
\hypersetup{pagebackref,colorlinks=true,linkcolor=blue,urlcolor=blue,breaklinks=true}
\usepackage{color}
\usepackage[onehalfspacing]{setspace}
\usepackage{tabularx}
\usepackage{amsfonts}
\usepackage{paralist}
\usepackage{aliascnt}
\usepackage[initials, lite]{amsrefs}
\usepackage{amscd}
\usepackage{blkarray}
\usepackage{mathbbol}
\usepackage{setspace}
\usepackage[inner=2.4cm,outer=2.4cm, bottom=3.2cm]{geometry}
\usepackage{tikz, tikz-cd}
\usepackage{calligra,mathrsfs}

\usepackage{tikz}
\usetikzlibrary{matrix}
\usetikzlibrary{arrows,calc}
\allowdisplaybreaks

\BibSpec{collection.article}{%
	+{}  {\PrintAuthors}                {author}
	+{,} { \textit}                     {title}
	+{.} { }                            {part}
	+{:} { \textit}                     {subtitle}
	+{,} { \PrintContributions}         {contribution}
	+{,} { \PrintConference}            {conference}
	+{}  {\PrintBook}                   {book}
	+{,} { }                            {booktitle}
	+{,} { }                            {series}
	+{, vol.} { }                            {volume}
	+{,} { }                            {publisher}
	+{,} { \PrintDateB}                 {date}
	+{,} { pp.~}                        {pages}
	+{,} { }                            {status}
	+{,} { \PrintDOI}                   {doi}
	+{,} { available at \eprint}        {eprint}
	+{}  { \parenthesize}               {language}
	+{}  { \PrintTranslation}           {translation}
	+{;} { \PrintReprint}               {reprint}
	+{.} { }                            {note}
	+{.} {}                             {transition}
	+{}  {\SentenceSpace \PrintReviews} {review}
}
\AtBeginDocument{%
	\def\MR#1{}
}

\makeatletter
\@namedef{subjclassname@2020}{%
	\textup{2020} Mathematics Subject Classification}
\makeatother

\newcommand{\Bl}{{\mathcal{B}\ell}}

\newcommand{\kk}{\mathbb{k}}

\newcommand{\sA}{A}
\newcommand{\sB}{B}

\newcommand{\sR}{\mathscr{R}}

\newcommand{\sT}{\mathscr{T}}

\newcommand{\fG}{\mathfrak{G}}
\newcommand{\sH}{\mathscr{H}}
\newcommand{\sM}{\mathscr{M}}

\newcommand{\length}{{\rm length}}
\newcommand{\NN}{\normalfont\mathbb{N}}
\newcommand{\ZZ}{\mathbb{Z}}

\newcommand{\MM}{{\normalfont\mathfrak{M}}}
\newcommand{\PP}{{\normalfont\mathbb{P}}}

\newcommand{\dd}{\mathbf{d}}

\newcommand{\mm}{{\normalfont\mathfrak{m}}}

\newcommand{\bv}{\mathbf{v}}
\newcommand{\fC}{\mathfrak{C}}

\newcommand{\pp}{{\normalfont\mathfrak{p}}}

\newcommand{\bn}{{\normalfont\mathbf{n}}}

\newcommand{\bm}{{\normalfont\mathbf{m}}}

\newcommand{\Ker}{\normalfont\text{Ker}}
\newcommand{\Coker}{\normalfont\text{Coker}}

\newcommand{\Ann}{\normalfont\text{Ann}}
\newcommand{\Supp}{\normalfont\text{Supp}}

\newcommand{\Sym}{\normalfont\text{Sym}}
\newcommand{\Rees}{\mathscr{R}}

\newcommand{\ee}{{\normalfont\mathbf{e}}}

\newcommand{\OO}{\mathcal{O}}

\newcommand{\FF}{\mathcal{F}}
\newcommand{\HL}{\normalfont\text{H}_{\mm}}
\newcommand{\HH}{\normalfont\text{H}}

\newcommand{\Proj}{\normalfont\text{Proj}}

\newcommand{\Spec}{\normalfont\text{Spec}}

\newcommand{\multProj}{\normalfont\text{MultiProj}}

\newcommand{\biProj}{{\normalfont\text{BiProj}}}

\DeclareMathOperator{\HT}{ht}
\DeclareMathOperator{\grade}{grade}

\DeclareMathOperator{\gr}{gr}

\def\f0{\mathbf{0}}

\def\fv{\mathbf{v}}

\def\ft{\mathbf{t}}

\def\fn{\mathbf{n}}
\def\1{\mathbf{1}}

\newtheorem{theorem}{Theorem}[section]

\newtheorem{headthm}{Theorem}

\newaliascnt{headcor}{headthm}

\aliascntresetthe{headcor}

\newaliascnt{headconj}{headthm}

\aliascntresetthe{headconj}

\newaliascnt{corollary}{theorem}
\newtheorem{corollary}[corollary]{Corollary}
\aliascntresetthe{corollary}

\newaliascnt{claim}{theorem}

\aliascntresetthe{claim}

\newaliascnt{lemma}{theorem}
\newtheorem{lemma}[lemma]{Lemma}
\aliascntresetthe{lemma}

\newaliascnt{conjecture}{theorem}

\aliascntresetthe{conjecture}

\newaliascnt{proposition}{theorem}
\newtheorem{proposition}[proposition]{Proposition}
\aliascntresetthe{proposition}

\theoremstyle{definition}
\newaliascnt{definition}{theorem}
\newtheorem{definition}[definition]{Definition}
\aliascntresetthe{definition}

\newaliascnt{notation}{theorem}
\newtheorem{notation}[notation]{Notation}
\aliascntresetthe{notation}

\newaliascnt{example}{theorem}
\newtheorem{example}[example]{Example}
\aliascntresetthe{example}

\newaliascnt{examples}{theorem}

\aliascntresetthe{examples}

\newaliascnt{remark}{theorem}
\newtheorem{remark}[remark]{Remark}
\aliascntresetthe{remark}

\newaliascnt{question}{theorem}

\aliascntresetthe{question}

\newaliascnt{questions}{theorem}

\aliascntresetthe{questions}

\newaliascnt{problem}{theorem}

\aliascntresetthe{problem}

\newaliascnt{construction}{theorem}

\aliascntresetthe{construction}

\newaliascnt{setup}{theorem}
\newtheorem{setup}[setup]{Setup}
\aliascntresetthe{setup}

\newaliascnt{setupdef}{theorem}

\aliascntresetthe{setupdef}

\newaliascnt{algorithm}{theorem}

\aliascntresetthe{algorithm}

\newaliascnt{observation}{theorem}

\aliascntresetthe{observation}

\newaliascnt{defprop}{theorem}

\aliascntresetthe{defprop}

\def\equationautorefname~#1\null{(#1)\null}
\def\sectionautorefname~#1\null{Section #1\null}
\def\subsectionautorefname~#1\null{\S #1\null}

\def\surjects{\twoheadrightarrow}

\DeclareFontFamily{OT1}{pzc}{}
\DeclareFontShape{OT1}{pzc}{m}{it}{<-> s * [1.100] pzcmi7t}{}
\DeclareMathAlphabet{\mathchanc}{OT1}{pzc}{m}{it}

\DeclareMathOperator{\fSpec}{\mathchanc{Spec}}

\DeclareMathOperator{\br}{{\rm br}}

\title{Relative mixed multiplicities and mixed Buchsbaum-Rim multiplicities}

\author{Yairon Cid-Ruiz}
\address{Department of Mathematics, North Carolina State University, Raleigh, NC, USA}
\email{ycidrui@ncsu.edu}

\date{\today}

\begin{document}

	\maketitle

	\begin{abstract}
		We define and study the natural multigraded extension of the relative multiplicities introduced by Simis, Ulrich and Vasconcelos.
		We call these new invariants \emph{relative mixed multiplicities}.
		We show that they have a stable value equal to the mixed Buchsbaum--Rim multiplicity of Kleiman and Thorup.
		Furthermore, we prove that integral dependence and birationality can be detected via the vanishing of relative mixed multiplicities.
	\end{abstract}

	\section{Introduction}

	Let $A \subseteq B$ be an inclusion of standard $\NN^p$-graded algebras over a Noetherian local ring $R = [A]_{(0,\ldots,0)} = [B]_{(0,\ldots,0)}$.
	Assume that $\length_R\left([B]_{\ee_i}/[A]_{\ee_i}\right) < \infty$ for all $1 \le i \le p$.
	Our main goal is to define and study a set of new invariants that we call \emph{relative mixed multiplicities} and that provide the natural multigraded extension of the \emph{relative multiplicities} introduced in seminal work of  Simis, Ulrich and Vasconcelos \cite{SUV_MULT}.
	In the $\NN$-graded setting (i.e., when $p = 1$), they introduced a series of relative multiplicities $e_t(A,B)$, one for each $t \ge 1$, associated to the numerical functions $\length_R\left([B]_n \big/  [A]_{n-t+1}[B]_{t-1}\right)$.
	They also showed that the relative multiplicities $e_t(A,B)$ have a stable value $e_\infty(A, B)$, and that the vanishing of $e_{\infty}(A,B)$ detects integrality and the vanishing of $e_1(A, B)$ detects both integrality and birationality.
	The relative multiplicity $e_1(A, B)$ was considered before by Kirby and Rees \cite{KIRBY_REES} (including a multigraded setting), whereas $e_\infty(A, B)$ was introduced by Kleiman and Thorup \cite{KLEIMAN_THORUP_GEOM} under the name of \emph{generalized Buchsbaum--Rim multiplicity}.
	
	\medskip
	
	In this paper, we extend the above invariants and results to a multigraded setting. 
	Our first main result is the following theorem that serves as the stepping stone to define the notion of relative mixed multiplicities.

	\begin{headthm}[\autoref{thm_relative_mixed_mult}]
		\label{thmA}
		Fix a tuple $\ft = (t_1,\ldots,t_p) \in \ZZ_+^p$ of positive integers.
		Then the following function
		$$
		\lambda_\ft^{A,B}(n_1,\ldots,n_p) \;=\; {\rm length}_R\left(
		\frac{\left[\sB\right]_{(n_1,\ldots,n_p)}} {\left[\sA\right]_{(n_1-t_1+1,\ldots,n_p-t_p+1)}\left[\sB\right]_{(t_1-1,\ldots,t_p-1)}}
		\right)
		$$
		coincides with a polynomial $P_{\ft}^{A,B}(n_1,\ldots,n_p)$ for all $n_i \gg 0$.
		Moreover, $P_{\ft}^{A,B}$ has total degree at most $\dim\left(\multProj(B)\right)$ and its normalized leading coefficients are nonnegative integers.
	\end{headthm}

	Let $X = \multProj(B)$ and $r = \dim(X)$.
	The polynomial $P_\ft^{A,B}$ from \autoref{thmA} can be written as 
	\[
	P_\ft^{A,B}(n_1,\ldots,n_p) \;=\;  \sum_{\beta \in \NN^p,\, |\beta|=r} e_\ft\left(\beta; A, B\right) \, \frac{n_1^{\beta_1}\cdots n_p^{\beta_p}}{\beta_1!\cdots\beta_p!} \;+\; \text{\rm(lower degree terms)}.
	\]
	We say that the nonnegative integers $e_\ft\left(\beta; A, B\right) \ge 0$ are the \emph{relative mixed multiplicities of $B$ over $A$}.
	We show that these new invariants are nonincreasing on the parameter $\ft \in \NN^p$ and that they have a stable value.
	Indeed, from \autoref{lem_properties}, we obtain the inequality $e_{\ft'}(\beta;A,B) \le e_{\ft}(\beta;A,B)$ for any $\ft' = (t_1',\ldots,t_p') \in \ZZ_+^p$ with $\ft' \ge \ft$ and the existence of the following limit
	$$
	e_\infty(\beta; A,B) \;=\; \lim_{t_1\to \infty, \ldots, t_p \to \infty} e_{(t_1,\ldots,t_p)}\left(\beta; A, B\right).
	$$
	As it turns out, this stable value coincides with the \emph{mixed Buchsbaum--Rim multiplicity} of Kleiman and Thorup \cite{KLEIMAN_THORUP_MIXED}:
	\[
	e_\infty(\beta; A,B) \;=\; \br_\beta\left([A]_{\ee_1},\ldots,[A]_{\ee_p};B\right)  \;=\; \int_P  c_1\left(\OO_P(\ee_1)\right)^{\beta_1}\cdots c_1\left(\OO_P(\ee_p)\right)^{\beta_p} \,\cap \, \left[P\right]_r;
	\]
	here $P = P_X^{Z_1,\ldots,Z_p}$ denotes the Kleiman--Thorup transform of $X = \multProj(B)$ with respect to $Z_1,\ldots,Z_p$, where $Z_i = V_+\left([A]_{\ee_i}B\right) \subset X$  is the closed subscheme  determined by $[A]_{\ee_i}$ (see \autoref{def_KT_transform} and \autoref{def_mixed_BR} for further details).
	
	\medskip
	
	Our proof of \autoref{thmA} is based upon using two main ingredients.
	Firstly, we consider the stable value $e_\infty(\beta; A, B)=\br_\beta([A]_{\ee_1},\ldots,[A]_{\ee_p}; B)$, which is geometrical by nature. 
	For that, we extend the joint blow-up construction of Kleiman and Thorup to an arbitrary multigraded setting in \autoref{sect_KT}, and this allows us to define the mixed Buchsbaum--Rim multiplicity $\br_\beta([A]_{\ee_1},\ldots,[A]_{\ee_p}; B)$.
	Secondly, in \autoref{sect_new_joint_BL}, we identify and introduce new joint blow-up constructions that measure the difference between each relative mixed multiplicity $e_\ft(\beta; A, B)$ and the stable value $e_\infty(\beta; A, B)$.
	Therefore, our approach is to define $e_\ft(\beta; A, B)$ as a sum of intersection numbers (multidegrees) in terms of the joint blow-up construction of Kleiman and Thorup and our new joint blow-up construction. 
	More precisely,  in \autoref{lem_properties}, we show the equality 
	$$
	e_\ft\left(\beta; A, B\right) \;=\; \br_\beta\left([A]_{\ee_1}, \ldots, [A]_{\ee_p}; B\right) \;+\; j_{\beta}^{\#}\big([A]_{\ee_1}, \ldots, [A]_{\ee_p}; \sH_{\ft-\mathbf{1}}\big),
	$$
	where $j_{\beta}^{\#}([A]_{\ee_1}, \ldots, [A]_{\ee_p}; \sH_{\ft-\mathbf{1}})$ is a mixed multiplicity of one of our new joint blow-up constructions considered over the ideal $\sH_{\ft-\mathbf{1}} = \big([B]_{(t_1-1,\ldots,t_p-1)}\big) \subset B$  (see \autoref{def_j_mult_new_joint} for further details).
	
	\medskip
	
	Similarly to the $\NN$-graded case, we obtain that the distinguished relative mixed multiplicities $e(\beta; A, B)  = e_{(1,\ldots,1)}(\beta; A, B)$ and $e_\infty(\beta; A, B)$ can detect integral dependence and birationality. 
	Our second main result is the following theorem.

	\begin{headthm}[\autoref{thm_criteria}]
		\label{thmB}
		Consider the associated morphism $$
		f \;:\; U \;\subseteq\; X = \multProj(B) \;\longrightarrow\; Y = \multProj(A)
		$$
		where $U = X \setminus V_+(A_{++}B).$
		Then the following statements hold: \smallskip
		\begin{enumerate}[\rm (i)]
			\item If we obtain a finite morphism $f : X \rightarrow Y$, then $e_\infty\left(\beta; A, B\right) = 0$ for all $\beta \in \NN^p$ with $|\beta| = r$.
			The converse holds if $B$ is equidimensional and catenary.
			\item If we obtain a finite birational morphism $f:X\rightarrow Y$, then $e\left(\beta; A, B\right) = 0$ for all $\beta \in \NN^p$ with $|\beta| = r$.
			The converse holds if $B$ is equidimensional and catenary.
		\end{enumerate}
	\end{headthm}
	
	To the best of our knowledge, even in the $\NN$-graded case, the result of \autoref{thmB} seems to be new  (see \autoref{rem_weaker_cond} and \autoref{cor_applications_NN}): for the converse statements of \autoref{thmB}, it is typically assumed that $B$ is equidimensional and \emph{universally} catenary (see \cite{SUV_MULT, KLEIMAN_THORUP_GEOM, KATZ_RED_MOD}).

	\medskip
	
	\noindent
	\textbf{Outline.} 
	The basic outline of this paper is as follows. 	
	In \autoref{sect_KT}, we extend the joint blow-up construction of Kleiman and Thorup to an arbitrary multigraded setting.
	We present some basic results regarding a multigraded version of $j$-multiplicities in  \autoref{sect_mixed_j_mult}.
	Our new joint blow-up construction is introduced in \autoref{sect_new_joint_BL}.
	We prove \autoref{thmA} in \autoref{sect_relative_mult}.
	Finally, \autoref{sect_applications} contains the proof of \autoref{thmB} and some applications.

	\medskip
	
	\noindent
	\textbf{Notation.} 
	Let $p \ge 1$ be a positive integer. 
	If $\bn = (n_1,\ldots,n_p),\bm = (m_1,\ldots,m_p) \in \ZZ^p$ are two multi-indexes, we write $\bn \ge \bm$ whenever $n_i \ge m_i$ for all $1 \le i \le p$, and $\bn > \bm$ whenever $n_j > m_j$ for all $1 \le j \le p$.
	For each $1 \le i \le p$, let $\ee_i \in \NN^p$ be the $i$-th elementary vector $\ee_i:=\left(0,\ldots,1,\ldots,0\right)$.
	Let $\mathbf{0} \in \NN^p$ and $\mathbf{1} \in \NN^p$ be the vectors $\mathbf{0}:=(0,\ldots,0)$ and $\mathbf{1}:=(1,\ldots,1)$ of $p$ copies of $0$ and $1$, respectively.  
	For any $\bn = (n_1,\ldots,n_p) \in \ZZ^p$, we define its weight as $\lvert \bn \rvert := n_1+\cdots+n_p$. 
	We say that an $\NN^p$-graded ring $B$ is \emph{standard} if $B$ is generated as an algebra over $[B]_{(0,\ldots,0)}$ by finitely many elements of the form $b \in [B]_{\ee_i}$ for some $1 \le i \le p$.
	For a standard $\NN^p$-graded ring $B$,  we set the following: 
	\begin{itemize}[--]
		\item $B_+ := \bigoplus_{\fv > \mathbf{0}} [B]_\fv$ is the ideal of elements of positive total degree.
		\item $B_{++} := \bigoplus_{\fv \ge \mathbf{1}} [B]_\fv$ is the multigraded irrelevant ideal.
		\item $\multProj(B) := \big\lbrace \pp \in \Spec(B) \mid \text{$\pp$ is $\NN^p$-graded and $\pp \not\supseteq B_{++}$} \big\rbrace$ is the corresponding multiprojective scheme.
		\item $V_+(I) := \big\lbrace \pp \in \multProj(B) \mid \pp \supseteq I \big\rbrace$ for any $\NN^p$-graded ideal $I \subset B$.
	\end{itemize}

	\section{An extension of a joint blow-up construction of Kleiman and Thorup}
	\label{sect_KT}
	
	Here, we provide a multigraded extension of a joint blow-up construction of Kleiman and Thorup \cite{KLEIMAN_THORUP_MIXED}, and we closely follow their approach.
	As a consequence, we can introduce the notion of mixed Buchsbaum-Rim multiplicities in an arbitrary multigraded setting.
	Throughout this section the following setup is fixed.
	
	\begin{setup}
		\label{setup_kleiman_thorup}
		Let $\sB$ be a standard $\NN^p$-graded algebra over a Noetherian local ring $R = \left[\sB\right]_{(0,\ldots,0)}$.
		Denote by $\mm \in R$ the maximal ideal of $R$.
		Let $X = \multProj(\sB) \subset \PP_R^{m_1} \times_R \cdots \times_R \PP_R^{m_p}$ be the multiprojective scheme determined by $\sB$, embedded in a product of projective spaces over $R$.
		Let $q \ge 1$ be a positive integer.
		For each $1 \le i \le q$, let $\dd_i \in \NN^p$ be a tuple of nonnegative integers and $H_i \subseteq [\sB]_{\dd_{i}}$ be a nonzero $R$-submodule.
		Let $Z_i \subset X$ be the closed subscheme determined by $H_i$ and $\mathcal{J}_i \subset \OO_X$ be the corresponding ideal sheaf.
		Let $Z \subset X$ be the closed subscheme defined by the ideal sheaf $\mathcal{J} = \mathcal{J}_1 \cdots \mathcal{J}_q \subset \OO_X$.
	\end{setup}

	Let $\widehat{X} := \fSpec_X\left(\OO_X[t]\right) = X \times_{\Spec(R)} \mathbb{A}_R^1$ be the relative affine line over $X$. 
	We view $X$ as a closed subscheme embedded in $\widehat{X}$ as the principal divisor $\{t=0\}$.
	Let $\widehat{\mathcal{B}} := \Bl_{Z_1,\ldots,Z_q}\big(\widehat{X}\big)$ be the joint blow-up of $\widehat{X}$ along the closed subschemes $Z_1,\ldots,Z_q$.
	This scheme can be obtained by successively blowing-up $\widehat{X}$ along  $Z_1, \ldots,Z_q$.
	It arises from the $\NN^q$-graded quasi-coherent $\OO_{\widehat{X}}$-algebra
	$$
	\sR = \sR\big((\mathcal{J}_1,t), \ldots, (\mathcal{J}_q,t)\big) \;:=\; \bigoplus_{(n_1,\ldots,n_q) \in \NN^q} (\mathcal{J}_1,t)^{n_1} \cdots (\mathcal{J}_q,t)^{n_q};
	$$
	i.e., the multi-Rees algebra of the ideal sheaves $(\mathcal{J}_1,t), \ldots, (\mathcal{J}_q,t) \subset \OO_{\widehat{X}}$.
	As we wish to embed $\widehat{\mathcal{B}}$ in a product of projective spaces, we consider the ``twisted''  $\NN^q$-graded algebra 
	$$
	\sT \;:=\; \bigoplus_{(n_1,\ldots,n_q) \in \NN^q} (\mathcal{J}_1,t)^{n_1} \cdots (\mathcal{J}_q,t)^{n_q} \,\otimes_{\OO_{\widehat{X}}} \, \OO_{\widehat{X}}\left(n_1 \dd_1 + \cdots + n_q \dd_q\right).
	$$
	The algebras $\sR$ and $\sT$ yield the same scheme $\widehat{\mathcal{B}}$ (see \cite[Lemma II.7.9]{HARTSHORNE} for the $\NN$-graded case).
	Since the sheaf $(\mathcal{J}_i, t)(\dd_i)$ is generated by global sections, we have a surjection $\OO_{\widehat{X}}^{r_i+1} \surjects (\mathcal{J}_i, t)(\dd_i)$ for some $r_i \ge 0$.
	In turn, this naturally gives the surjection
	$$
	\Sym_{\OO_{\widehat{X}}}\left(\OO_{\widehat{X}}^{r_1+1}\right) \otimes_{\OO_{\widehat{X}}} \cdots \otimes_{\OO_{\widehat{X}}} \Sym_{\OO_{\widehat{X}}}\left(\OO_{\widehat{X}}^{r_q+1}\right) \;\surjects\; \sT.
	$$
	Consequently, we obtain a closed immersion $\iota : \widehat{\mathcal{B}} \hookrightarrow \PP_{\widehat{X}}^{r_1} \times_{\widehat{X}} \cdots \times_{\widehat{X}} \PP_{\widehat{X}}^{r_q} \cong  \widehat{X} \times_R  \PP_{R}^{r_1} \times_{R} \cdots \times_{R} \PP_{R}^{r_q}$.
	Then we set 
	$$
	\OO_{\widehat{\mathcal{B}}}(\ee_i) \,:=\, \iota^*\OO_{\PP_{\widehat{X}}^{r_1} \times_{\widehat{X}} \cdots \times_{\widehat{X}} \PP_{\widehat{X}}^{r_q}}(\ee_i)
	$$
	for any $\ee_i \in \NN^q$ with $1 \le i \le q$. 
	Thus our choice of $\OO_{\widehat{\mathcal{B}}}(\ee_i)$ equals the corresponding tautological line bundle determined by $\sR$ but twisted by $\OO_{\widehat{X}}(\dd_i)$.
	
	By an abuse of notation, we shall write $\OO_{X} (\ee_i)$ for some $\ee_i \in \NN^p$ with $1 \le i \le p$ and $\OO_{\widehat{\mathcal{B}}} (\ee_j)$ for some $\ee_j \in \NN^q$ with $1 \le j \le q$.
	Thus, depending on the context,  $\ee_i$ could denote an elementary basis vector either in $\NN^p$ or in $\NN^q$.
	
	Let $M$ be a finitely generated $\ZZ^p$-graded $\sB$-module and $\mathscr{M} := \widetilde{M}$ be the corresponding coherent sheaf on $X$.
	Let $\widehat{\mathscr{M}} := \mathscr{M}[t]$ be the $\OO_{\widehat{X}}$-module given as the pullback of $\mathscr{M}$.
	Denote by $\widehat{\mathcal{B}}(\sM)$ the coherent $\OO_{\widehat{\mathcal{B}}}$-module associated to the following $\NN^q$-graded  $\sT$-module
	$$
	T(\mathscr{M}) \;:=\;  \bigoplus_{(n_1,\ldots,n_q) \in \NN^q} (\mathcal{J}_1,t)^{n_1} \cdots (\mathcal{J}_q,t)^{n_q} \widehat{\mathscr{M}} \;\otimes_{\OO_{\widehat{X}}} \, \OO_{\widehat{X}}\left(n_1 \dd_1 + \cdots + n_q \dd_q\right), 
	$$
	which is a quasi-coherent $\OO_{\widehat{X}}$-module.
	
	The main objects of interest in this section are introduced in the following definition.
	
	\begin{definition}
		\label{def_KT_transform}
		Let $P = P_X^{Z_1,\ldots,Z_q} := \widehat{\mathcal{B}} \times_{\widehat{X}} Z$ be the restriction of $\widehat{\mathcal{B}}$ to the preimage of $Z$.
		Let $\pi : P \rightarrow X$ be the natural projection and $j : P \hookrightarrow \widehat{\mathcal{B}}$ be the natural closed immersion.
		We say that the \emph{Kleiman--Thorup transform} of $\sM$ with respect to $Z_1,\ldots,Z_q$ is  the coherent $\OO_P$-module given by
		\[
		P(\sM) = P_X^{Z_1,\ldots,Z_q}(\sM) \;:=\; j^*\left(\widehat{\mathcal{B}}(\mathscr{M})\right). 
		\]	
	\end{definition}
	
	The lemma below provides a basic dimension computation that will be needed (also, see  \cite[\S 3]{KLEIMAN_THORUP_GEOM}, \cite[Chapter 5]{huneke2006integral}).
	
	\begin{lemma}
		\label{lem_dim_KT}
		We have the inequality $\dim\left(\Supp(P(\sM))\right) \le \dim\left(\Supp(\sM)\right)$.
	\end{lemma}
	\begin{proof}
		Let $U = \Spec(A) \subset X$ be an affine open subscheme and $N$ be a finitely generated $A$-module with $\mathscr{M}{\mid_U} = \widetilde{N}$.
		Let $J_i \subset A$ be the ideal corresponding to $U \cap Z_i$, and set $K_i = (J_i, t)A[t]$.
		Consider the module $L= \Rees_{N[t]}\left(K_1,\ldots,K_q\right) := (N[t])\left[K_1T_1,\ldots,K_qT_q\right]$
		over the multi-Rees algebra $\Rees_{A[t]}\left(K_1,\ldots,K_q\right) := (A[t])\left[K_1T_1,\ldots,K_qT_q\right]$ with standard $\NN^q$-grading induced by setting $\deg(t)  =\mathbf{0} \in \NN^q$ and $\deg(T_i) = \ee_i \in \NN^q$.
		Notice that $P(\sM){\mid_{\pi^{-1}(U)}}$ is the sheaf associated to the $\ZZ^q$-graded module $L/(J_1\cdots J_q, t)L$.
		Since  $\dim\left(L/(J_1\cdots J_q,t)L\right) \le \dim(L) - 1 \le \dim(N) + q$, it follows that $\dim\big(\Supp(P(\sM){\mid_{\pi^{-1}(U)}})\big) \le \dim(N)$ (see, e.g., \cite[\S 1]{HYRY_MULTIGRAD}, \cite[\S 3.1]{SPECIALIZATION_ARON}), as required.
	\end{proof}
	
	We have the following theorem that allows us to define mixed Buchsbaum--Rim multiplicities. 
	It extends the results of \cite[\S 8]{KLEIMAN_THORUP_MIXED} to an arbitrary multigraded setting.
	
	\begin{remark}
		As in the work of Kleiman and Thorup \cite{KLEIMAN_THORUP_GEOM, KLEIMAN_THORUP_MIXED}, the intersection numbers stated below are well-defined by utilizing the rudiments of intersection theory developed by Thorup in \cite{THORUP_INT_TH}, which are suitable to work over our base Noetherian local ring $R$. 
		This treatment of Thorup extends the necessary results from Fulton's book \cite{FULTON_INTERSECTION_THEORY}.
	\end{remark}

	\begin{theorem}
		\label{thm_poly_KT}
		Assume \autoref{setup_kleiman_thorup}.
		Let $M$ be a finitely generated $\ZZ^p$-graded $\sB$-module and $\sM = \widetilde{M}$ be the corresponding coherent $\OO_X$-module.
		Set $r = \dim(\Supp(\sM))$, and suppose that $Z \cap \Supp(\sM)$ is contained in the closed fiber $X \times_{\Spec(R)} \Spec(R/\mm)$ of $X$.
		Then the following function
		$$
		\lambda_M^{H_1,\ldots,H_q}(v_1,\ldots,v_p,n_1,\ldots,n_q) \;=\; \length_R\left(
		\frac{[M]_{(v_1,\ldots,v_p) + n_1\dd_1 + \cdots + n_q\dd_q}}
		{H_1^{n_1}\cdots H_q^{n_q} \, [M]_{(v_1,\ldots,v_p)}}\right)
		$$
		coincides with a polynomial $P_M^{H_1,\ldots,H_q}(v_1,\ldots,v_p,n_1,\ldots,n_q)$ for all $v_i \gg 0$ and $n_i \gg 0$. 
		Moreover, $P_M^{H_1,\ldots,H_q}$ has total degree at most $r$ and we can write
		\[
		P_M^{H_1,\ldots,H_q}(v_1,\ldots,v_p,n_1,\ldots,n_q) \;=\;  \sum_{\substack{\alpha\in \NN^p, \beta \in \NN^q\\|\alpha|+|\beta|=r}} e(\alpha, \beta) \, \frac{v_1^{\alpha_1}\cdots v_p^{\alpha_p}n_1^{\beta_1}\cdots n_q^{\beta_q}}{\alpha_1!\cdots\alpha_p!\beta_1!\cdots\beta_q!} \;+\; \text{\rm(lower degree terms)}
		\]
		where the normalized coefficient $e(\alpha,\beta)$ equals the intersection number 
		\[
		\int_P c_1\left(\pi^*(\OO_X(\ee_1))\right)^{\alpha_1}\cdots c_1\left(\pi^*(\OO_X(\ee_p))\right)^{\alpha_p}\, c_1\left(\OO_P(\ee_1)\right)^{\beta_1}\cdots c_1\left(\OO_P(\ee_q)\right)^{\beta_q} \,\cap \, \left[P(\sM)\right]_r.
		\]
	\end{theorem}
	\begin{proof}
		From our conventions above, we can identify $P$ as a closed subscheme of a product of projective spaces over $R$ via the following composition of natural closed immersions
		\begin{align*}
			f \;:\; P \;\hookrightarrow\; \widehat{\mathcal{B}} \times_{\Spec(R[t])} \Spec\left(R[t]/(t)\right) \;\hookrightarrow\;  X \times_R&  \PP_{R}^{r_1} \times_{R} \cdots \times_{R} \PP_{R}^{r_q} \\ &\;\hookrightarrow\; \PP := \PP_{R}^{m_1} \times_{R} \cdots \times_{R} \PP_{R}^{m_p} \times_{R} \PP_{R}^{r_1} \times_{R} \cdots \times_{R} \PP_{R}^{r_q}.
		\end{align*}
		For any $\fv = (v_1,\ldots,v_p) \in \NN^p$ and $\fn = (n_1,\ldots,n_q) \in \NN^q$, the pullback of $\OO_\PP(\fv,\fn)$ to $P$ is given by 
		$$
		f^*\big(\OO_{\PP}(\fv, \fn)\big) \;\cong\;  \pi^*\big(\OO_X(\fv)\big) \,\otimes_{\OO_P}\, \OO_{P}(\fn).
		$$
		To simplify notation, let $\FF := P(\sM)$, $\FF(\bn) := P(\sM) \otimes_{\OO_P} \OO_{P}(\bn)$ and 
		$$
		\FF(\fv,\fn) \;:=\; P(\sM) \otimes_{\OO_P} f^*\big(\OO_{\PP}(\fv, \fn)\big) \;\cong\; P(\sM) \otimes_{\OO_P} \pi^*\big(\OO_X(\fv)\big) \otimes_{\OO_P} \OO_{P}(\fn)
		$$
		for any $\fv \in \NN^p$ and $\fn \in \NN^q$.
		From \cite[Lemma 4.3]{KLEIMAN_THORUP_GEOM}, it follows that the function 
		\begin{align}
			\label{eq_Hilb_poly_H0}
			\begin{split}
				\chi\big(f_*(\FF)(\fv,\fn)\big) &= \sum_i {(-1)}^i \length_R\big(\HH^i\left(\PP, f_*(\FF)(\fv,\fn)\right) \big) \\
				&= \sum_i {(-1)}^i \length_R\big(\HH^i\left(P, \FF(\fv,\fn) \right)\big)
			\end{split}		
		\end{align}
		equals a polynomial $P_\FF(\bv, \bn)$ of total degree equal to $\dim(\Supp(\FF)) \le r = \dim(\Supp(\sM))$ (see \autoref{lem_dim_KT}).
		For $\alpha \in \NN^p, \beta \in \NN^q$ with $|\alpha| + |\beta| = r$, the coefficient of $\fv^\alpha\fn^\beta$ in $P_\FF(\fv,\fn)$ is equal to $\frac{1}{\alpha_1!\cdots\alpha_p! \beta_1! \cdots \beta_q!}$ times the intersection number 
		\[
		\int_\PP \,
		c_1\left(\OO_\PP(\ee_1,\mathbf{0})\right)^{\alpha_1}\cdots c_1\left(\OO_\PP(\ee_p,\mathbf{0})\right)^{\alpha_p} \, c_1\left(\OO_\PP(\mathbf{0},\ee_1)\right)^{\beta_1}\cdots c_1\left(\OO_\PP(\mathbf{0},\ee_q)\right)^{\beta_q} \,\cap \, \left[f_*(\FF)\right]_r;
		\]
		where, depending on the context, we have either $\mathbf{0} = (0,\ldots,0) \in \NN^p$ or $\mathbf{0} = (0,\ldots,0) \in \NN^q$.
		By the projection formula, this latter intersection number coincides with $e(\alpha,\beta)$.

		We need two observations regarding the natural projection $\pi : P \subset \PP_{X}^{r_1} \times_X \cdots \times_X \PP_{X}^{r_q} \rightarrow X$.
		We set $\overline{T}(\sM) := T(\sM) \otimes_{\OO_{\widehat{X}}} \OO_Z$.
		For any affine open subscheme $U \subset X$, by utilizing \cite[Lemma 4.2]{KLEIMAN_THORUP_GEOM}, we obtain that $\pi_*\left(\FF(\bn)\right){\mid_U} \cong \left[\overline{T}(\sM)\mid_U\right]_\bn$ and $R^i\pi_*\left(\FF(\bn)\right){\mid_U} = 0$ for all $\bn \gg \mathbf{0}$ and $i \ge 1$.
		Hence we globally get 
		$$
		\pi_*\left(\FF(\bn)\right) \;\cong\; \left[\overline{T}(\sM)\right]_\bn \qquad \text{ and }    \qquad R^i\pi_*\left(\FF(\bn)\right) = 0
		$$
		for all $\bn \gg \mathbf{0}$ and $i \ge 1$.
		So, the Leray spectral sequence $$E_2^{i,j} = \HH^i\left(X, R^j\pi_*(\FF(\fv,\fn))\right) \;\Rightarrow\; \HH^{i+j}\left(P, \FF(\fv,\fn)\right)$$
		degenerates and yields the isomorphism 
		\begin{equation}
			\label{eq_projection_H0}
			\HH^i\big(P, \FF(\fv, \bn)\big) \;\cong\; \HH^i\big(X, \left[\overline{T}(\sM)\right]_\bn(\fv)\big) 
		\end{equation}
		for all $i \ge 0$, $\bv \in \NN^p$ and  $\bn \gg \mathbf{0} \in \NN^q$.  
		
		We now study the algebraic counterpart of $\overline{T}(\sM)$.
		Given a $\ZZ^p$-graded $\sB$-module $N$, we define the $\ZZ^p$-graded $\sB$-module
		$$
		\overline{T}(N) \;:=\;  \bigoplus_{(n_1,\ldots,n_q) \in \NN^q} \frac{(H_1,t)^{n_1} \cdots (H_q,t)^{n_q} N[t]}{(H_1\cdots H_q,t)(H_1,t)^{n_1} \cdots (H_q,t)^{n_q}N[t]}\; \otimes_{\sB} \, \sB\left(n_1 \dd_1 + \cdots + n_q \dd_q\right),
		$$
		where $\deg(t) = \mathbf{0} \in \NN^p$.
		We consider the Noetherian $\ZZ^p$-graded algebra $S := \overline{T}(\sB)$.
		Notice that $S$ is a quotient of the multi-Rees algebra $\Rees_{B[t]}\left((H_1,t), \ldots, (H_q,t)\right) := (B[t])\big[(H_1,t)T_1,\ldots,(H_q,t)T_q\big]$ with $\ZZ^p$-grading induced by the grading of $B$ and by setting $\deg(T_i) = -\dd_i \in \ZZ^p$ and $\deg(t) = \mathbf{0} \in \NN^p$.
		Then $\overline{T}(M)$ is a finitely generated $\ZZ^p$-graded $S$-module.
		Notice that $B_{++} S = S_{++} := \bigoplus_{\fv \ge \mathbf{1}} [S]_\fv$, where $\mathbf{1} = (1,\ldots,1) \in \NN^p$.
		Therefore, we get the vanishing $\big[\HH_{B_{++}}^i\big(\overline{T}(M)\big)\big]_\fv \cong \big[\HH_{S_{++}}^i\big(\overline{T}(M)\big)\big]_\fv=0$ for all $i \ge 0$ and $\fv \gg \mathbf{0}$ (this follows verbatim to the $\ZZ$-graded result from \cite[Theorem 2.1]{CHARDIN_POWERS}).
		By utilizing \cite[Corollary 1.5]{HYRY_MULTIGRAD}, we obtain 
		\begin{equation}
			\label{eq_H0_X_to_T}
			\HH^0\big(X, \overline{T}(\sM)(\fv)\big) \;\cong \left[\overline{T}(M)\right]_\fv \quad \text{ and } \quad \HH^i\big(X, \overline{T}(\sM)(\fv)\big) = 0
		\end{equation}
		for all $\fv \gg \mathbf{0}$ and $i \ge 1$.
		
		By combining \autoref{eq_Hilb_poly_H0}, \autoref{eq_projection_H0} and \autoref{eq_H0_X_to_T}, we get tuples of integers $\fv' \in \NN^p$ and $\fn' \in \NN^q$, such that  		
		\begin{equation}
			\label{eq_poly_F_eq_E}
			P_\FF(\bv, \bn) \;=\; \length_R\left(E_{\fv,\fn}\right)
		\end{equation}
		for all $\fv \ge \fv'$ and $\fn \ge \fn'$, where 
		$$
		E_{\fv,\fn} \;:=\; \left[\frac{(H_1,t)^{n_1} \cdots (H_q,t)^{n_q} M[t]}{(H_1\cdots H_q,t)(H_1,t)^{n_1} \cdots (H_q,t)^{n_q}M[t]}\right]_{\fv + n_1\dd_1 + \cdots + n_q\dd_q}.
		$$
		We write 
		$$
		(H_1,t)^{n_1} \cdots (H_q,t)^{n_q} M[t] \;=\; \sum_{0 \le j_i \le n_i} H_1^{j_1}\cdots H_q^{j_q}\,t^{|\bn|-(j_1+\cdots+j_q)}\,M[t] \;=\; \sum_{k = 0}^{|\bn|} L_k t^{|\bn|-k}M[t],
		$$
		where $L_k$ is the $R$-module $\sum_{\substack{0 \le j_i \le n_i\\ j_1+\cdots+j_q = k}} H_1^{j_1}\cdots H_q^{j_q}$; by standard convention, we have $H_i^0 = R = [B]_\mathbf{0}$.
		Since $H_1\cdots H_q L_k M\subset L_{k+1}M$ for $0 \le k < |\bn|$, we get the equality
		$$
		(H_1\cdots H_q,t)(H_1,t)^{n_1} \cdots (H_q,t)^{n_q}M[t] \;=\; \Big(H_1\cdots H_q L_{|\bn|} + \sum_{k=1}^{|\bn|}L_kt^{|\bn|+1-k} + t^{|\bn|+1}\Big) M[t].
		$$
		Therefore, we obtain the isomorphism of $R$-modules
		$$
		E_{\fv,\fn} \;\cong\; \left[\frac{L_{|\bn|}M}{H_1\cdots H_qL_{|\bn|}M} \,\;\;\bigoplus\;\;\, \bigoplus_{k=0}^{|\bn|-1} \frac{L_{k}M}{L_{k+1}M}\right]_{\fv + n_1\dd_1 + \cdots + n_q\dd_q}
		$$
		and consequently the equality 
		\begin{equation}
			\label{eq_length_E}
			\length_R\left(E_{\fv,\fn}\right) = {\rm length}_R\left(
			\frac{[M]_{\fv + n_1\dd_1 + \cdots + n_q\dd_q}}
			{H_1^{n_1+1}\cdots H_q^{n_q+1} \, [M]_{\fv-\dd_1-\cdots-\dd_q}}\right) = \lambda_M^{H_1,\ldots,H_q}(\fv-\dd_1-\cdots-\dd_q, \fn+\mathbf{1}).
		\end{equation}
		Finally, \autoref{eq_poly_F_eq_E} and \autoref{eq_length_E} yield the equality $\lambda_M^{H_1,\ldots,H_q}(\fv, \fn) = P_\FF(\fv+\dd_1+\cdots+\dd_q, \fn-\mathbf{1})$ for all $\fv \ge \fv'-\dd_1-\cdots -\dd_q$ and $\fn \ge \fn'+\mathbf{1}$.
		The polynomials $P_\FF(\fv+\dd_1+\cdots+\dd_q, \fn-\mathbf{1})$ and $P_\FF(\fv, \fn)$ have the same top degree part.
		Therefore, by setting $P_M^{H_1,\ldots,H_q}(\fv,\fn):=P_\FF(\fv+\dd_1+\cdots+\dd_q, \fn-\mathbf{1})$, the proof of the theorem is complete.
	\end{proof}

	Finally, after having the result of \autoref{thm_poly_KT}, we can make the following definition.
	
	\begin{definition}
		\label{def_mixed_BR}
		Under the notation and assumptions of \autoref{thm_poly_KT}, for any $\alpha\in\NN^p$ and  $\beta \in \NN^q$ with $|\alpha| + |\beta| \ge r$, we say that the \emph{mixed Buchsbaum--Rim multiplicity of $M$ of type $(\alpha,\beta)$ with respect to $H_1,\ldots,H_q$} is the following nonnegative integer
		$$
		\br_{\alpha,\beta}\left(H_1,\ldots,H_q; M\right) \;:=\; 
		\begin{cases}
			e(\alpha,\beta) & \;\; \text{ if } |\alpha| + |\beta| = r \\		
			0 & \;\; \text{ if } |\alpha| + |\beta| > r. 
		\end{cases}
		$$
		Notice that $\br_{\alpha,\beta}\left(H_1,\ldots,H_q; M\right)$ can be defined either in geometrical terms as an intersection number associated to the Kleiman--Thorup transform $P(\sM)$ or in algebraic fashion by utilizing the Hilbert-like function $\lambda_M^{H_1,\ldots,H_q}$. 
		For any $\beta \in \NN^p$ with $|\beta| \ge r$, we also define 
		$$
		\br_\beta(H_1,\ldots,H_q; M) \;:=\; \br_{\mathbf{0},\beta}\left(H_1,\ldots,H_q; M\right) \;=\; 	\int_P  c_1\left(\OO_P(\ee_1)\right)^{\beta_1}\cdots c_1\left(\OO_P(\ee_q)\right)^{\beta_q} \,\cap \, \left[P(\sM)\right]_r.
		$$
	\end{definition}

\section{Mixed $j$-multiplicities}
\label{sect_mixed_j_mult}	
	
In this short section, we quickly develop some basic results regarding a multigraded version of \emph{$j$-multiplicities} (see \cite{ACHILLES_MANARESI_J_MULT}, \cite[\S 6.1]{FLENNER_O_CARROLL_VOGEL}).
These results will be needed in the next section.
The setup below is used throughout this section.

\begin{setup}
	Let $(R, \mm)$ be a Noetherian local ring and $B$ be a standard $\NN^p$-graded $R$-algebra.
	Set $X = \multProj(B)$.
\end{setup}

Let $M$ be a finitely generated $\ZZ^p$-graded $B$-module and $\sM := \widetilde{M}$ be the corresponding coherent $\OO_X$-module.
Recall that $\dim(\sM) = \dim\big(M/\HH_{B_{++}}^0(M)\big)-p$ (see, e.g., \cite[\S 1]{HYRY_MULTIGRAD}, \cite[\S 3.1]{SPECIALIZATION_ARON}).
Notice that $\HL^0(M)$ is finitely generated over $B/\mm^kB$ for $k \gg 0$, and this latter ring is a standard $\NN^p$-graded algebra over the Artinian local ring $R/\mm^k$.
Then we may define the mixed $j$-multiplicities of $M$ as the mixed multiplicities of $\HL^0(M)$.
For details on the notion of \emph{mixed multiplicities}, the reader is referred to \cite{MIXED_MULT,HERMANN_MULTIGRAD}.
Fix $\beta \in \NN^p$ with $|\beta| \ge \dim\big(\Supp(\sM)\big)$.
	
\begin{definition}
	The \emph{mixed $j$-multiplicity of $M$ of type $\beta$} is given by 
	$
	j_\beta(M) := e\big(\beta;\, \HL^0(M)\big).
	$
\end{definition}
	
	We have the following expected results (cf., \cite[\S 6.1]{FLENNER_O_CARROLL_VOGEL}).
	
\begin{lemma}
	\label{lem_j_mult_add}
	Let $0 \rightarrow M' \rightarrow M \rightarrow M'' \rightarrow 0$ be a short exact sequence of finitely generated $\ZZ^p$-graded $B$-modules. 
	Then we have the equality
	$
	j_\beta(M) \;=\; j_\beta(M') + j_\beta(M'').
	$
\end{lemma}
\begin{proof}
The proof follows similarly to \cite[Proposition 6.1.2]{FLENNER_O_CARROLL_VOGEL}.
We have an exact sequence $0 \rightarrow \HL^0(M') \rightarrow \HL^0(M) \rightarrow \HL^0(M'')$.
Let $Q := \Coker\left(\HL^0(M) \rightarrow \HL^0(M'')\right)$.
For any minimal prime $\pp$ of $M$ that belongs to the closed fiber $X \times_{\Spec(R)} \Spec(R/\mm)$, we obtain $M_\pp' = \HL^0(M_\pp')$, $M_\pp = \HL^0(M_\pp)$ and $M_\pp'' = \HL^0(M_\pp'')$.
Hence $\dim(\Supp(\widetilde{Q})) < \dim(\Supp(\sM))$, and the result follows from the additivity of mixed multiplicities.
\end{proof}
	
\begin{corollary}
	We have that $j_\beta(M) = j_\beta\big(M/\HH_{B_{++}}^0(M)\big)$.
\end{corollary}	
\begin{proof}
	Since $j_\beta\big(\HH_{B_{++}}^0(M)\big) = 0$, the result follows from \autoref{lem_j_mult_add}.
\end{proof}

\begin{lemma}
	\label{lem_assoc_j_mult}
	We have the equality 
	$$
	j_\beta(M) \;=\; \sum_{\pp} \length_{B_\pp}(M_\pp)\cdot e(\beta; B/\pp)
	$$
	where $\pp$ runs over the minimal primes of $M$ that belong to the closed fiber $X \times_{\Spec(R)} \Spec(R/\mm)$ and satisfy $\dim(B/\pp) = |\beta| + p$.
\end{lemma}	
\begin{proof}
	The proof is similar to \cite[Proposition 6.1.3]{FLENNER_O_CARROLL_VOGEL}.
	However we now use the additivity result of \autoref{lem_j_mult_add}.
\end{proof}

	\section{New joint blow-up constructions}
	\label{sect_new_joint_BL}
	
	This section introduces new joint blow-up constructions that will be our second main ingredient to define relative mixed multiplicities.
	The setup below is used throughout this section.
	
	\begin{setup}
		\label{setup_joint_blowup}
		Let $\sB$ be a standard $\NN^p$-graded algebra over a Noetherian local ring $R = \left[\sB\right]_{(0,\ldots,0)}$ with maximal ideal $\mm \subset R$.
		For each $1 \le i \le p$, let $H_i \subseteq [\sB]_{\ee_{i}}$ be a nonzero $R$-submodule, and $I_i = \left(H_i\right) \subset \sB$ be the ideal generated by $H_i$.
		Let $I = I_1 + \cdots + I_p \subset \sB$ be the ideal given as the sum of the ideals $I_i$.
		Let $X = \multProj(B)$.
	\end{setup}

	We consider the multi-Rees algebra 
	$$
	\sR \;=\; \sR\left(I_1,\ldots,I_p\right) \;:=\; \bigoplus_{(n_1,\ldots,n_p) \in \NN^p} I_1^{n_1} \cdots I_p^{n_p} \; T_1^{n_1}\cdots T_p^{n_p} \;\subset\; \sB[T_1,\ldots,T_p],
	$$
	where $T_1,\ldots,T_p$ are new variables.
	We endow $\sR$ with a standard $(\NN^p \oplus \NN^p)$-grading.
	For any homogeneous element $b \in \sB$, we set 
	$
	\deg_\sR(b) = \deg_\sB(b) \oplus \mathbf{0} \in \NN^p \oplus \NN^p,
	$
	and for any nonzero element $h \in H_i$, we set 
	$
	\deg_\sR(hT_i) = \mathbf{0} \oplus \ee_i \in \NN^p \oplus \NN^p.
	$
	This implies that $\sR$ is generated by elements of total degree one, and thus it is standard $(\NN^p \oplus \NN^p)$-graded as an $R$-algebra.
	
	Let $M$ be a finitely generated $\ZZ^p$-graded $\sB$-module and $\sM := \widetilde{M}$ be the corresponding coherent $\OO_X$-module.
	Then we obtain the following finitely generated $(\ZZ^p \oplus \NN^p)$-graded $\sR$-module 
	$$
	\sR_M \;=\; \sR\left(I_1,\ldots,I_p; M\right) \;:=\; \bigoplus_{(n_1,\ldots,n_p) \in \NN^p} I_1^{n_1} \cdots I_p^{n_p} M \; T_1^{n_1}\cdots T_p^{n_p} \;\subset\; M[T_1,\ldots,T_p].
	$$
	For any $\bv = (v_1,\ldots,v_p) \in \ZZ^p$ and $\bn = (n_1,\ldots,n_p) \in \NN^p$, the corresponding $(\bv,\bn)$-graded part of $\sR_M$ is given by
	$$
	\left[\sR_M\right]_{\bv, \bn} \;=\; \left[I_1^{n_1}\cdots I_p^{n_p} M\right]_{\bv + \bn}.
	$$ 
	We now introduce our new joint blow-up constructions. 
	We take the Rees algebra of the ideal $I\sR$ over the previously defined multi-Rees algebra $\sR$, that is:
	$$
	\fC \;=\; \fC(I_1,\ldots,I_p) \;:=\; \sR\left(I\sR\right) \;=\; \bigoplus_{k=0}^\infty \, I^k\sR\, T^k  \;\subset\; \sR[T].
	$$
	We put a standard $(\NN^p \oplus \NN^p)$-grading on $\fC$ by naturally extending the one of $\sR$.
	We say that $\fC$ has an ``internal grading'' induced by $\sR$.
	For any homogeneous element $w \in \sR$, we set $\deg_\fC(w) = \deg_\sR(w)$, and for any nonzero element $h \in H_i \subset I$, we set 
	$
	\deg_\fC(hT) = \deg_\sR(h).
	$
	Hence $\fC$ is generated by elements of total degree one, and so it is also standard $(\NN^p \oplus \NN^p)$-graded as an $R$-algebra.
	The corresponding associated graded ring 
	$$
	\fG \;=\; \fG(I_1,\ldots,I_p) \;:=\; \gr_{I\sR}(\sR) \;=\; \fC / I \fC \;=\; \bigoplus_{k = 0}^\infty \, \frac{I^k\sR}{I^{k+1} \sR} 
	$$
	also obtains a standard $(\NN^p \oplus \NN^p)$-grading as an $R$-algebra.
	We have the following finitely generated $(\ZZ^p \oplus \NN^p)$-graded modules
	$$
	\fC_M \;=\; \fC(I_1,\ldots,I_p;M) \;:=\;  \sR\left(I\sR;\sR_M\right) \;=\; \bigoplus_{k=0}^\infty \, I^k\sR_M T^k  \;\subset\; \sR_M[T]
	$$
	and 
	$$
	\fG_M \;=\; \fG(I_1,\ldots,I_p;M) \;:=\; \fC_M / I \fC_M \;=\; \bigoplus_{k = 0}^\infty \, \frac{I^k\sR_M}{I^{k+1} \sR_M} 
	$$
	over $\fC$ and $\fG$, respectively.

Let $\mathfrak{X} = \multProj(\fG)$ be the multiprojective scheme determined by $\fG$ (which can naturally be embedded in a product of the form $\PP_R^{m_1} \times_R \cdots \times_R \PP_R^{m_{2p}}$), and $\widetilde{\fG_M}$ be the corresponding coherent $\OO_\mathfrak{X}$-module.

\begin{lemma}
	\label{lem_dim_GM}
	We have the inequality $\dim(\Supp(\widetilde{\fG_M})) \le \dim(\Supp(\sM))$.
\end{lemma}
\begin{proof}
	Let $\overline{M} = M / \HH_{B_{++}}^0(M)$.
	Notice that the $(\fv,\fn)$-graded part of $\fG_M$ is given by 
			\begin{align*}
		\big[\fG_{M}\big]_{\fv,\bn}  &\;=\; \bigoplus_{k=0}^\infty \left[\frac{I^k\sR_{M}}{I^{k+1} \sR_{M}} \right]_{\fv,\bn} \\
		&\;=\; \bigoplus_{k=0}^\infty \left[\frac{I^k I_1^{n_1}\cdots I_p^{n_p}M}{I^{k+1} I_1^{n_1}\cdots I_p^{n_p}M} \right]_{\fv+\bn} 
		\;=\; \bigoplus_{k=0}^{|\fv|-{\rm indeg}(M)} \left[\frac{I^k I_1^{n_1}\cdots I_p^{n_p}M}{I^{k+1} I_1^{n_1}\cdots I_p^{n_p}M} \right]_{\fv+\bn} 
	\end{align*}
	where ${\rm indeg}(M) := \inf\big\{ v \in \ZZ \mid [M]_\fv \neq 0 \text{ for some $\fv \in \ZZ^p$ with $|\fv| = v$}\big\}$.
	Since $[M]_\fv \cong \left[\overline{M}\right]_\fv$ for all $\fv \gg \mathbf{0}$, it follows that $\big[\fG_{M}\big]_{\fv,\bn} \cong \big[\fG_{\overline{M}}\big]_{\fv,\bn}$ for all $\fv \gg \mathbf{0}$ and $\fn \gg \mathbf{0}$.
	Therefore $\widetilde{\fG_M} \cong \widetilde{\fG_{\overline{M}}}$.
	We have  $\dim(\fG_{\overline{M}}) \le \dim(\overline{M}) + p$ (see, e.g, \cite[Chapter 5]{huneke2006integral}).
	Hence we obtain
	$$
	\dim(\Supp(\widetilde{\fG_M})) \;\le\; \dim(\fG_{\overline{M}})-2p \;\le\; \dim(\overline{M}) - p \;=\; \dim(\Supp(\sM)),
	$$
	as required.
\end{proof}

As a consequence of the above dimension computation, we can introduce the following invariants. 
Similar invariants were introduced in the $\NN$-graded case by Ulrich and Validashti \cite{ULRICH_VALIDASHTI_INT_DEP_MOD}.  

\begin{definition}
	\label{def_j_mult_new_joint}
	For any $\alpha, \beta \in \NN^p$ with $|\alpha| + |\beta| \ge \dim(\Supp(\sM))$, we consider the following mixed $j$-multiplicities
	$$
	j_{\alpha,\beta}\big(H_1,\ldots,H_p; M\big) \;:=\;  j_{\alpha,\beta}(\fG_M) \quad \text{ and } \quad j_{\alpha,\beta}^{\#}\big(H_1,\ldots,H_p; M\big) \;:=\;  j_{\alpha,\beta}(B_+\fG_M);
	$$
	see \autoref{sect_mixed_j_mult}.
	
	\noindent
	We also set $j_{\beta}\big(H_1,\ldots,H_p; M\big) := j_{\mathbf{0},\beta}\big(H_1,\ldots,H_p; M\big)$ and $j_{\beta}^{\#}\big(H_1,\ldots,H_p; M\big) := j_{\mathbf{0},\beta}^{\#}\big(H_1,\ldots,H_p; M\big)$ for any $\beta \in \NN^p$ with $|\beta| \ge \dim(\Supp(\sM))$.
\end{definition}

The next remark tells us that under the conditions needed to define relative mixed multiplicities, we have that the mixed $j$-multiplicities of $B_+\fG_M$ are just mixed multiplicities.
 
\begin{remark}
	\label{rem_finite_length}
	Assume $\length_R\left([B]_{\ee_i}/H_i\right) < \infty$ for all $1 \le i \le p$. 
	Then $\overline{\fG} := \fG/(0:_{\fG}B_+\fG)$ is a standard $(\NN^p \oplus \NN^p)$-graded algebra over the Artinian local ring $\Big[\overline{\fG}\Big]_{\mathbf{0}, \mathbf{0}} = R/\mathfrak{a}$ with $\mathfrak{a} = \bigcap_{i=1}^p \Ann_R\left([B]_{\ee_i}/H_i\right)$.
	Notice that $B_+\fG_M$ is always a finitely generated $(\ZZ^p \oplus \NN^p)$-graded module over $\overline{\fG}$.
	Therefore, we obtain that $j_{\alpha,\beta}^{\#}\big(H_1,\ldots,H_p; M\big) = e(\alpha, \beta; B_+\fG_M)$.
\end{remark}

	The following theorem shows the existence of a polynomial that will be essential to our approach to defining relative mixed multiplicities.
	This existence result is a byproduct of considering the above joint blow-up constructions. 
	The leading coefficients of this polynomial will coincide with the mixed multiplicities of $B_+\fG_\sH$ for a certain ideal $\sH \subset B$.

	\begin{theorem}
		\label{thm_new_joint_Bl}
		Assume \autoref{setup_joint_blowup} and that $\length_R\left([B]_{\ee_i}/H_i\right) < \infty$ for all $1 \le i \le p$.
		Set $r = \dim(X)$. 
		Fix a tuple $\ft = (t_1,\ldots,t_p) \in \NN^p$ of nonnegative integers, and let $\sH_\ft \subset \sB$ be the ideal generated by the elements in the graded part $[\sB]_\ft$.
		Then the following function
		$$
		\lambda_\ft^{H_1,\ldots,H_p}(v_1,\ldots,v_p,n_1,\ldots,n_p) = {\rm length}_R\left(\frac{{H_1}^{n_1}\cdots {H_p}^{n_p}\,[\sB]_{(v_1,\ldots,v_p)}}{H_1^{v_1+n_1-t_1}\cdots H_p^{v_p+n_p-t_p} \, [\sB]_{(t_1,\ldots,t_p)}}\right)
		$$
		coincides with a polynomial $P_\ft^{H_1,\ldots,H_p}(v_1,\ldots,v_p,n_1,\ldots,n_p)$ for all $v_i \gg 0$ and $n_i \gg 0$.
		Moreover, $P_\ft^{H_1,\ldots,H_p}$ has total degree at most $r$ and we can write
		\[
		P_\ft^{H_1,\ldots,H_p}(v_1,\ldots,v_p,n_1,\ldots,n_p) \;=\;  \sum_{\substack{\alpha\in \NN^p, \beta \in \NN^p\\|\alpha|+|\beta|=r}} e\left(\alpha, \beta\right) \, \frac{v_1^{\alpha_1}\cdots v_p^{\alpha_p}n_1^{\beta_1}\cdots n_p^{\beta_p}}{\alpha_1!\cdots\alpha_p!\beta_1!\cdots\beta_p!} \;+\; \text{\rm(lower degree terms)}
		\]
		where $e(\alpha,\beta):=j_{\alpha,\beta}^{\#}\big(H_1,\ldots,H_p; M\big)=e\left(\alpha, \beta; B_+\fG_{\sH_\ft}\right)$ is the mixed multiplicity of $B_+\fG_{\sH_\ft}$ of type $(\alpha,\beta)$.
	\end{theorem}
	\begin{proof}	
		We may assume that $\fv > \ft$.
		Let $c = |\fv| - |\ft|$.
		We study the finitely generated $(\NN^p \oplus \NN^p)$-graded $\fG$-module $B_+\fG_{\sH_\ft}$.	
		From the above conventions, the $(\fv,\bn)$-graded part is given by 
		\begin{align*}
			\big[B_+\fG_{\sH_\ft}\big]_{\fv,\bn}  &\;=\; \bigoplus_{k=0}^\infty \left[B_+ \cdot \frac{I^k\sR_{\sH_\ft}}{I^{k+1} \sR_{\sH_\ft}} \right]_{\fv,\bn} \\
			&\;=\; \bigoplus_{k=0}^\infty \left[B_+ \cdot \frac{I^k I_1^{n_1}\cdots I_p^{n_p}\sH_\ft}{I^{k+1} I_1^{n_1}\cdots I_p^{n_p}\sH_\ft} \right]_{\fv+\bn} \\
			&\;=\; \bigoplus_{k=0}^{c-1} \left[B_+ \cdot \frac{I^k I_1^{n_1}\cdots I_p^{n_p}\sH_\ft}{I^{k+1} I_1^{n_1}\cdots I_p^{n_p}\sH_\ft} \right]_{\fv+\bn} \\
			&\;=\; \bigoplus_{k=0}^{c-1} \left[\frac{I^k I_1^{n_1}\cdots I_p^{n_p}\sH_\ft}{I^{k+1} I_1^{n_1}\cdots I_p^{n_p}\sH_\ft} \right]_{\fv+\bn}.
		\end{align*}
		Thus, we obtain the equality 
		$$
		{\rm length}_R\left(\big[B_+\fG_{\sH_\ft}\big]_{\fv,\bn}\right) \;=\; {\rm length}_R\left(
		\frac{\left[I_1^{n_1}\cdots I_p^{n_p}\sH_\ft\right]_{\fv + \bn}}{\left[I^cI_1^{n_1}\cdots I_p^{n_p}\sH_\ft\right]_{\fv + \bn}}\right).
		$$
		It is clear that 
		$$
		\left[I_1^{n_1}\cdots I_p^{n_p}\sH_\ft\right]_{\fv + \bn} \;=\; H_1^{n_1}\cdots H_p^{n_p} [\sB]_\fv.
		$$
		On the other hand, since 
		$$
		I^cI_1^{n_1}\cdots I_p^{n_p} \;=\; \sum_{j_1+\cdots+j_p=c} I_1^{n_1+j_1}\cdots I_p^{n_p+j_p}
		$$
		and $\big[I_1^{n_1+j_1}\cdots I_p^{n_p+j_1}\sH_\ft\big]_{\fv + \bn} = 0$ when any $j_i > v_i -t_i$, by the pigeonhole principle it follows that 
		$$
		\left[I^cI_1^{n_1}\cdots I_p^{n_p}\sH_\ft\right]_{\fv + \bn} \;=\; H_1^{v_1+n_1-t_1}\cdots H_p^{v_p+n_p-t_p} [\sB]_\ft.
		$$
		Then the result follows from the existence of (multigraded) Hilbert polynomials (see, e.g., \cite[Lemma 4.3]{KLEIMAN_THORUP_GEOM}, \cite[Theorem 4.1]{HERMANN_MULTIGRAD}, \cite[Theorem 3.4]{MIXED_MULT}).
		Indeed, there is a polynomial $P_\ft^{H_1,\ldots,H_q}(\fv,\fn) := P_{B_+\fG_{\sH_\ft}}(\fv, \fn)$ of total degree equal to $\dim(\Supp(\widetilde{B_+\fG_{\sH_\ft}}))$	such that 
		$$
		P_\ft^{H_1,\ldots,H_p}(\fv,\fn) \;=\; P_{B_+\fG_{\sH_\ft}}(\fv, \fn) \;=\; {\rm length}_R\big(\big[B_+\fG_{\sH_\ft}\big]_{\fv,\bn}\big) \;=\; \lambda_\ft^{H_1,\ldots,H_p}(\fv,\fn)
		$$
		for all $\fv \gg \mathbf{0}$ and $\fn \gg \mathbf{0}$.
		Finally, we have the upper bound $\dim(\Supp(\widetilde{B_+\fG_{\sH_\ft}})) \le r = \dim(X)$ (see \autoref{lem_dim_GM}).
	\end{proof}
	
	We now show an additivity result for the mixed $j$-multiplicities $j_{\alpha,\beta}\big(H_1,\ldots,H_p; M\big)$. 
	Our proof is inspired by the arguments used in the proofs of \cite[Theorem 1.2.6]{FLENNER_O_CARROLL_VOGEL} and \cite[Proposition 6.1.7]{FLENNER_O_CARROLL_VOGEL}.
	We also benefited from the exposition in \cite{ULRICH_VALIDASHTI_INT_DEP_MOD}.

	\begin{proposition}
		\label{prop_additive}
		Let $0 \rightarrow M' \rightarrow M \rightarrow M'' \rightarrow 0$ be a short exact sequence of finitely generated $\ZZ^p$-graded $B$-modules.
		Set $r = \dim(\Supp(\sM))$.
		Then, for all $\alpha, \beta \in \NN^p$  with $|\alpha|+|\beta| \ge r$, we have the equality 
		$$
		j_{\alpha,\beta}\big(H_1,\ldots,H_p; M\big) \;=\; 	j_{\alpha,\beta}\big(H_1,\ldots,H_p; M'\big) + j_{\alpha,\beta}\big(H_1,\ldots,H_p; M''\big).
		$$
	\end{proposition}
	\begin{proof}
		We consider the extended Rees algebra 
		$$
		\Rees^+(I\Rees) \;:=\; \Rees\left[I\Rees T, T^{-1}\right] \;=\; \bigoplus_{k \in \ZZ} I^k\sR T^k \;\subset\; \Rees[T,T^{-1}].
		$$		
		By following our grading conventions, we obtain that $\Rees^+(I\Rees)$ is a standard $(\NN^p \oplus \NN^p)$-graded algebra over $R[T^{-1}]$.
		We take the localization $\fC^+ :=  \Rees^+(I\Rees) \otimes_{R[T^{-1}]} R'$ with $R' := R[T^{-1}]_{(\mm, T^{-1})R[T^{-1}]}$.
		Therefore, we have that $\fC^+$ is a standard $(\NN^p \oplus \NN^p)$-graded algebra over the local ring $R'$, and so we are free to compute mixed $j$-multiplicities over $\fC^+$ (see \autoref{sect_mixed_j_mult}).
		We also consider the extended Rees-module 
		$$
		 \Rees^+(I\Rees; \Rees_M) \;:=\; \bigoplus_{k \in \ZZ} I^k\sR_M T^k \;\subset\; \Rees_M[T, T^{-1}]
		$$
		and take the localization $\fC_M^+ := \Rees^+(I\Rees; \Rees_M) \otimes_{R[T^{-1}]} R'$.
		
		Let $\mathscr{Y} = \multProj(\fC^+)$ be the multiprojective scheme associated to $\fC^+$.
		Let $N := \Ker\left(\fC_M^+ \rightarrow \fC_{M''}^+\right)$ and $L := N / \fC_{M'}^+$.
		We have the following commutative diagram with exact rows and columns of finitely generated $(\ZZ^p \oplus \NN^p)$-graded $\fC^+$-modules
		\begin{equation*}
			\begin{tikzpicture}[baseline=(current  bounding  box.center)]
				\matrix (m) [matrix of math nodes,row sep=3em,column sep=6.5em,minimum width=2em, text height=1.5ex, text depth=0.25ex]
				{
					& 0 & 0 & 0 &	\\			
					0 & N & \fC_{M}^+ & \fC_{M''}^+ & 0\\
					0 & N & \fC_{M}^+ & \fC_{M''}^+ & 0.\\
				};
				\path[-stealth]
				(m-2-1) edge (m-2-2)
				(m-2-2) edge (m-2-3)
				(m-2-3) edge (m-2-4)
				(m-2-4) edge (m-2-5)
				(m-3-1) edge (m-3-2)
				(m-3-2) edge (m-3-3)
				(m-3-3) edge (m-3-4)
				(m-3-4) edge (m-3-5)
				(m-1-2) edge (m-2-2)
				(m-1-3) edge (m-2-3)
				(m-1-4) edge (m-2-4)
				(m-2-2) edge node [left] {$T^{-1}$} (m-3-2)
				(m-2-3) edge node [left] {$T^{-1}$} (m-3-3)
				(m-2-4) edge node [left] {$T^{-1}$} (m-3-4)
				;		
			\end{tikzpicture}	
		\end{equation*}
		The snake lemma yields the short exact sequence 
		\begin{equation}
			\label{eq_ex_seq_N_gr_gr}
			0 \longrightarrow N/T^{-1}N \longrightarrow \fG_M \longrightarrow \fG_{M''} \longrightarrow 0.
		\end{equation}
		Let $U$ and $V$ be the kernel and cokernel of the multiplication map $L \xrightarrow{T^{-1}} L$, respectively.
		Hence we consider the exact sequence 
		\begin{equation}
			\label{eq_ex_seq_mult_L}
			0 \longrightarrow U \longrightarrow L \xrightarrow{\quad T^{-1} \quad} L \longrightarrow V \longrightarrow 0.
		\end{equation}
		We have the following commutative diagram with exact rows and columns 
		\begin{equation*}
			\begin{tikzpicture}[baseline=(current  bounding  box.center)]
				\matrix (m) [matrix of math nodes,row sep=3em,column sep=6.5em,minimum width=2em, text height=1.5ex, text depth=0.25ex]
				{
					&  &  & 0 &	\\			
					& 0 & 0 & U &	\\			
					0 & \fC_{M'}^+ &  N & L & 0\\
					0 & \fC_{M'}^+ &  N & L & 0\\
					&  &  & V &	\\
					&  &  & 0. &	\\
				};
				\path[-stealth]
				(m-3-1) edge (m-3-2)
				(m-3-2) edge (m-3-3)
				(m-3-3) edge (m-3-4)
				(m-3-4) edge (m-3-5)
				(m-4-1) edge (m-4-2)
				(m-4-2) edge (m-4-3)
				(m-4-3) edge (m-4-4)
				(m-4-4) edge (m-4-5)
				(m-2-2) edge (m-3-2)
				(m-2-3) edge (m-3-3)
				(m-2-4) edge (m-3-4)
				(m-3-2) edge node [left] {$T^{-1}$} (m-4-2)
				(m-3-3) edge node [left] {$T^{-1}$} (m-4-3)
				(m-3-4) edge node [left] {$T^{-1}$} (m-4-4)
				(m-1-4) edge (m-2-4)
				(m-4-4) edge (m-5-4)
				(m-5-4) edge (m-6-4)
				;		
			\end{tikzpicture}	
		\end{equation*}
		Again, by utilizing the snake lemma, we get the exact sequence
		\begin{equation}
			\label{eq_ex_seq_gr_A_N}
			0 \longrightarrow U \longrightarrow \fG_{M'} \longrightarrow N/T^{-1}N \longrightarrow V \longrightarrow 0.
		\end{equation}
		Since we have the following explicit description
		$$
		\frac{\Ker\big(\Rees^+(I\Rees; \Rees_M) \rightarrow \Rees^+(I\Rees; \Rees_{M''})\big)}{\Rees^+(I\Rees; \Rees_{M'})} \;\cong\; \bigoplus_{k \in \ZZ,\, n_1,\ldots,n_p \ge 0} \frac{I^kI_1^{n_1}\cdots I_p^{n_p}M \cap M'}{I^kI_1^{n_1}\cdots I_p^{n_p}M'},
		$$
		the multi-ideal version of the Artin-Rees lemma \cite[Theorem 17.1.6]{huneke2006integral} implies that $(T^{-1}I_1\cdots I_p)^c \cdot L = 0$ for some $c \ge 0$.
		Consider the coherent $\OO_\mathscr{Y}$-modules $\widetilde{L}$  and $\mathchanc{Q} = \widetilde{\fC_M^+/T^{-1}I_1\cdots I_p \fC_M^+}$.
		It then follows that 
		$$
		\dim(\Supp(\widetilde{L})) \;\le\; \dim(\Supp(\mathchanc{Q})) \;\le\; r = \dim(\Supp(\sM));
		$$
		we can proceed as in \autoref{lem_dim_GM}.

		Finally, by utilizing the additivity result of \autoref{lem_j_mult_add} and combining \autoref{eq_ex_seq_N_gr_gr}, \autoref{eq_ex_seq_mult_L} and  \autoref{eq_ex_seq_gr_A_N}, we obtain the required equality $j_{\alpha,\beta}(\fG_M) =	j_{\alpha,\beta}(\fG_{M'}) + j_{\alpha,\beta}(\fG_{M''})$.
	\end{proof}

\begin{corollary}
	\label{cor_torsion_my_j_mult}
		The following statements holds: 
	\begin{enumerate}[\rm (i)]
			\item $j_{\alpha,\beta}\big(H_1,\ldots,H_p; M\big) = j_{\alpha,\beta}\big(H_1,\ldots,H_p; \overline{M}\big)$ for all $\alpha, \beta \in \NN^p$ with $|\alpha|+|\beta| \ge \dim(\Supp(\sM))$, where $\overline{M} = M/\HH_{B_{++}}^0(M)$.
			\item $j_{\alpha,\beta}\big(H_1,\ldots,H_p; \sH_\ft\big) = j_{\alpha,\beta}\big(H_1,\ldots,H_p; B\big)$ for all $\ft, \alpha,\beta \in \NN^p$ with $|\alpha|+|\beta| \ge \dim(X)$.
	\end{enumerate}
\end{corollary}
\begin{proof}
	Both statements follow from \autoref{lem_dim_GM} and \autoref{prop_additive}.
\end{proof}	
	
	Finally, we have the following nonincreasing behavior for the mixed $j$-multiplicities $j_{\alpha,\beta}^{\#}$. 

\begin{corollary}
	\label{cor_ineq_j_sharp}
	Let $\ft, \ft' \in \NN^p$ such that $\ft' \ge \ft$.
	 Then, for all $\alpha,\beta \in \NN^p$ with $|\alpha|+|\beta| \ge \dim(X)$, we have
	 $$
	 j_{\alpha,\beta}^{\#}\big(H_1,\ldots,H_p; \sH_{\ft'}\big) \;\le\; j_{\alpha,\beta}^{\#}\big(H_1,\ldots,H_p; \sH_{\ft}\big).
	 $$
\end{corollary}
\begin{proof}
	By \autoref{cor_torsion_my_j_mult} and \autoref{lem_j_mult_add}, we may assume that $(0:_{B} B_{++}) = 0$ (i.e., $\grade(B_{++})>0$) and it suffices to show that 
	$$
	j_{\alpha,\beta}\big(\fG_{\sH_{\ft'}}/B_+\fG_{\sH_{\ft'}}\big) \;\ge \; j_{\alpha,\beta}\big(\fG_{\sH_{\ft}}/B_+\fG_{\sH_{\ft}}\big).
	$$
	Notice that $\fG_{\sH_{\ft}}/B_+\fG_{\sH_{\ft}} \cong \Rees(I,I_1,\ldots,I_p; \sH_{\ft}) / B_+\Rees(I,I_1,\ldots,I_p; \sH_{\ft})$ is a (relative) fiber cone and that we have the following explicit description
	$$
	\fG_{\sH_{\ft}}/B_+\fG_{\sH_{\ft}} \;=\; \bigoplus_{\fv,\fn\in \NN^p} \big[\fG_{\sH_{\ft}}/B_+\fG_{\sH_{\ft}}\big]_{\fv,\fn}  \;=\; \bigoplus_{\fv,\fn\in \NN^p, \fv \ge \ft} H_1^{v_1+n_1-t_1}\cdots H_p^{v_p+n_p-t_p} [\sB]_\ft;
	$$
	see the computations in the proof of \autoref{thm_new_joint_Bl}.
	Consider the standard $(\NN^p\oplus\NN^p)$-graded $R$-algebra $S:=\fG/B_+\fG$.
	By using the faithfully flat extension $R \rightarrow R'=R[z]_{\mm R[z]}$ and making the base change $B \rightarrow B \otimes_R R'$, we may assume that the residue field of $R$ is infinite.
	Then we can find an element $x \in [B]_{\ft'-\ft}$ which is a nonzerodivisor (see, e.g., \cite[Proposition 1.5.12]{BRUNS_HERZOG}), and so multiplication by $x$ yields the following inclusion 
			\begin{equation*}
		\begin{tikzpicture}[baseline=(current  bounding  box.center)]
			\matrix (m) [matrix of math nodes,row sep=3em,column sep=6.5em,minimum width=2em, text height=1.5ex, text depth=0.25ex]
			{
			   \Big(\fG_{\sH_{\ft}}/B_+\fG_{\sH_{\ft}}\Big)(\ft-\ft', \mathbf{0})	& 	\fG_{\sH_{\ft'}}/B_+\fG_{\sH_{\ft'}}\\			
			};
			\draw[right hook->] (m-1-1)--(m-1-2) node [midway,above] {$x$};	
			;		
		\end{tikzpicture}	
	\end{equation*}
	of finitely generated $(\NN^p \oplus \NN^p)$-graded $S$-modules.
	Finally, \autoref{lem_j_mult_add} yields the required inequality $j_{\alpha,\beta}(\fG_{\sH_{\ft}}/B_+\fG_{\sH_{\ft}}) \le j_{\alpha,\beta}(\fG_{\sH_{\ft'}}/B_+\fG_{\sH_{\ft'}})$.
\end{proof}

	\section{Relative mixed multiplicities}
	\label{sect_relative_mult}	
	
	In this section, we introduce the notion of relative mixed multiplicities. 
	Our existence proof follows by combining \autoref{thm_poly_KT} and \autoref{thm_new_joint_Bl}.
	The setup below is used during this section.

	\begin{setup}
		\label{setup_rel_mult}
		Let $A \subseteq B$ be an inclusion of standard $\NN^p$-graded algebras over a Noetherian local ring $R = [A]_{(0,\ldots0)} = [B]_{(0,\ldots0)}$.
		Assume that $\length_R\left([B]_{\ee_i}/[A]_{\ee_i}\right) < \infty$ for all $1 \le i \le p$.
		Let $X = \multProj(B)$ and $r = \dim(X)$.
	\end{setup}
	
	The theorem below is the main result of this section.
	
	\begin{theorem}	
		\label{thm_relative_mixed_mult}
		Assume \autoref{setup_rel_mult}.
		Fix a tuple $\ft = (t_1,\ldots,t_p) \in \ZZ_+^p$ of positive integers.
		Then the following function
		$$
		\lambda_\ft^{A,B}(n_1,\ldots,n_p) \;=\; {\rm length}_R\left(
		\frac{\left[\sB\right]_{(n_1,\ldots,n_p)}} {\left[\sA\right]_{(n_1-t_1+1,\ldots,n_p-t_p+1)}\left[\sB\right]_{(t_1-1,\ldots,t_p-1)}}
		\right)
		$$
		coincides with a polynomial $P_{\ft}^{A,B}(n_1,\ldots,n_p)$ for all $n_i \gg 0$.
		Moreover, $P_{\ft}^{A,B}$ has total degree at most $r$ and its normalized leading coefficients are nonnegative integers.
	\end{theorem}
	\begin{proof}
		The proof follows by using \autoref{thm_poly_KT}, \autoref{thm_new_joint_Bl} and the following short exact sequence 
		$$
		0 \;\longrightarrow\; \frac{[\sA]_\fn[\sB]_{\fv}}{[\sA]_{\fv+\fn-\ft+\mathbf{1}}  [\sB]_{\ft-\mathbf{1}}} \;\longrightarrow\; \frac{[B]_{\fv+\fn}}{[\sA]_{\fv+\fn-\ft+\mathbf{1}} [\sB]_{\ft-\mathbf{1}}} \;\longrightarrow\; \frac{[B]_{\fv + \fn}}{[\sA]_\fn[\sB]_{\fv}} \;\longrightarrow\; 0.
		$$
		From \autoref{thm_new_joint_Bl}, we obtain a polynomial $P_{\ft-\mathbf{1}}^{[A]_{\ee_1},\ldots,[A]_{\ee_p}}(\fv, \fn)$ of total degree at most $r$ that satisfies the equality 
		$$
		P_{\ft-\mathbf{1}}^{[A]_{\ee_1},\ldots,[A]_{\ee_p}}(\fv, \fn) \;=\; \length_R\left(\frac{[\sA]_\fn[\sB]_{\fv}}{[\sA]_{\fv+\fn-\ft+\mathbf{1}}  [\sB]_{\ft-\mathbf{1}}}\right)
		$$
		for all $\fv \gg \mathbf{0}$ and $\fn \gg \mathbf{0}$.
		Similarly, \autoref{thm_poly_KT} gives a polynomial $P_B^{[A]_{\ee_1},\ldots,[A]_{\ee_p}}(\fv, \fn)$ of total degree at most $r$ such that 
		$$
		P_B^{[A]_{\ee_1},\ldots,[A]_{\ee_p}}(\fv, \fn) \;=\;  \length_R\left(\frac{[B]_{\fv + \fn}}{[\sA]_\fn[\sB]_{\fv}}\right)
		$$
		for all $\fv \gg \mathbf{0}$ and $\fn \gg \mathbf{0}$.
		Consider the polynomial $P(\fv,\fn) := P_{\ft-\mathbf{1}}^{[A]_{\ee_1},\ldots,[A]_{\ee_p}}(\fv, \fn) + P_B^{[A]_{\ee_1},\ldots,[A]_{\ee_p}}(\fv, \fn)$.
		Then there exist vectors $\fv' \in \NN^p$ and $\fn' \in \NN^p$ such that 
		$$
		P(\fv,\fn) \;=\;  \length_R\left(\frac{[B]_{\fv+\fn}}{[\sA]_{\fv+\fn-\ft+\mathbf{1}} [\sB]_{\ft-\mathbf{1}}}\right)
		$$
		for all $\fv \ge \fv'$ and $\fn \ge \fn'$.
		By taking $P_{\ft}^{A,B}(\bn) := P(\fv', \fn - \fv')$, the result of the theorem follows.
	\end{proof}

	With \autoref{thm_relative_mixed_mult} in hand we are ready to introduce our new invariants: \emph{relative mixed multiplicities}.

	\begin{definition}
		Assume \autoref{setup_rel_mult}.
		Fix a tuple $\ft = (t_1,\ldots,t_p) \in \ZZ_+^p$ of positive integers, and let $P_{\ft}^{A,B}$ as in \autoref{thm_relative_mixed_mult}.
		Write 
		\[
		P_\ft^{A,B}(n_1,\ldots,n_p) \;=\;  \sum_{\beta \in \NN^p,\, |\beta|=r} e_\ft\left(\beta; A, B\right) \, \frac{n_1^{\beta_1}\cdots n_p^{\beta_p}}{\beta_1!\cdots\beta_p!} \;+\; \text{\rm(lower degree terms)}.
		\]
		We say that the nonnegative integer $e_\ft\left(\beta; A, B\right) \ge 0$ is the \emph{relative mixed multiplicity of $B$ over $A$ of type $(\ft, \beta)$}.
		For any $\beta \in \NN^p$ with $|\beta| > r$, we set $e_\ft(\beta; A, B) := 0$.
	\end{definition}	
	
	The next lemma shows that relative mixed multiplicities are nonincreasing on the parameter $\ft \in \ZZ^p$ and that they eventually stabilize with value equal to a mixed Buchsbaum-Rim multiplicity.
	
	\begin{lemma}
		\label{lem_properties}
		Let $\beta \in \NN^p$ with $|\beta| \ge r$.
		The following statements hold: 
		\begin{enumerate}[\rm (i)]
			\item $e_\ft\left(\beta; A, B\right) = \br_\beta\left([A]_{\ee_1}, \ldots, [A]_{\ee_p}; B\right) + j_{\beta}^{\#}\big([A]_{\ee_1}, \ldots, [A]_{\ee_p}; \sH_{\ft -\mathbf{1}}\big)$ for all $\ft \in \ZZ_+^p$.
			\item {\rm (Nonincreasing):} For any $\ft'\in \ZZ_+^p$ with $\ft' \ge \ft$, we have $e_{\ft'}\left(\beta; A, B\right) \le e_\ft\left(\beta; A, B\right)$.
			\item {\rm (Stabilization):} $$
			\lim_{t_1\to \infty, \ldots, t_p \to \infty} e_{(t_1,\ldots,t_p)}\left(\beta; A, B\right) \;=\; \br_\beta\left([A]_{\ee_1}, \ldots, [A]_{\ee_p}; B\right).$$
		\end{enumerate}
	\end{lemma}	
	\begin{proof}
		(i) This part is clear from the proof of \autoref{thm_relative_mixed_mult}.
		
		(ii) The result follows from part (i) and \autoref{cor_ineq_j_sharp}.
		
		(iii) By \autoref{thm_poly_KT} and \autoref{thm_relative_mixed_mult}, there is some $\fv' \in \NN^p$ such that, for all $\ft > \fv'$ and $\bn \gg \mathbf{0}$, we have the equalities
		$$
		P_B^{[A]_{\ee_1},\ldots, [A]_{\ee_p}}(\ft-\mathbf{1}, \bn)  \;=\; \length_R\left(\frac{[B]_{\fn+\ft-\mathbf{1}}}{[\sA]_\fn[\sB]_{\ft-\mathbf{1}}}\right) = P_\ft^{A, B}(\bn+\ft-\mathbf{1}).
		$$
		So, for all $\ft \ge \fv'$,  by comparing the leading coefficients of the polynomials $P_B^{[A]_{\ee_1},\ldots, [A]_{\ee_p}}(\ft-\mathbf{1}, \bn)$ and  $P_\ft^{A, B}(\bn+\ft-\mathbf{1})$ in $\bn$, the claim follows.
	\end{proof}	
	
	In our applications in the next section, the relative mixed multiplicities $e_{(1,\ldots,1)}(\beta; A,B)$ and the stable values will play an important role.
	Hence we introduce the following notation.
	
	\begin{notation}
		For all $\beta \in \NN^p$ with $|\beta| \ge r$, set 
		\[
		e(\beta; A,B) := e_{(1,\ldots,1)}(\beta; A,B) \qquad \text{ and } \qquad e_\infty\left(\beta; A, B\right) := \lim_{t_1\to \infty, \ldots, t_p \to \infty} e_{(t_1,\ldots,t_p)}\left(\beta; A, B\right).
		\]
	\end{notation}

	\section{Applications}
	\label{sect_applications}
	
	In this section, we use relative mixed multiplicities to provide integral dependence and birationality criteria in a multigraded setting. 
	These results show that relative mixed multiplicities extend the fact that relative multiplicities can detect integral dependence and birationality in an $\NN$-graded setting (see \cite{SUV_MULT}).
	Throughout this section we use the following setup. 
	
	\begin{setup}
		\label{setup_applications}
		Let $(R, \mm, \kappa)$ be a Noetherian local ring, and $A \subseteq B$ be an inclusion of standard $\NN^p$-graded algebras over $R = [A]_{(0,\ldots0)} = [B]_{(0,\ldots0)}$. 
		Assume that $\length_R\left([B]_{\ee_i}/[A]_{\ee_i}\right) < \infty$ for all $1 \le i \le p$.
	\end{setup}
	
	Our goal is to prove the following theorem.

	\begin{theorem}
		\label{thm_criteria}
		Assume \autoref{setup_applications}.
		Consider the associated morphism $$
		f \;:\; U \;\subseteq\; X = \multProj(B) \;\longrightarrow\; Y = \multProj(A)
		$$
		where $U = X \setminus V_+(A_{++}B).$
		Set $r = \dim(X)$.
		Then the following statements hold: \smallskip
		\begin{enumerate}[\rm (i)]
			\item If we obtain a finite morphism $f : X \rightarrow Y$, then $e_\infty\left(\beta; A, B\right) = 0$ for all $\beta \in \NN^p$ with $|\beta| = r$.
			The converse holds if $B$ is equidimensional and catenary.
			\item If we obtain a finite birational morphism $f:X\rightarrow Y$, then $e\left(\beta; A, B\right) = 0$ for all $\beta \in \NN^p$ with $|\beta| = r$.
			The converse holds if $B$ is equidimensional and catenary.
		\end{enumerate}
	\end{theorem}

\begin{remark}
	\label{rem_weaker_cond}
	We compare \autoref{thm_criteria} to known results in the $\NN$-graded case.
	To prove the converse statements of \autoref{thm_criteria}, one typically assumes that $B$ is equidimensional and \emph{universally} catenary.
	For instance, the approach in \cite{SUV_MULT} uses the associated graded ring $G = \gr_{A_+B}(B)$ as the main tool, which yields an elegant and unified treatment.
	However, if one considers $G$,  one needs $G$ to be equidimensional and catenary.
	Then, the standard assumption for $G$ to be equidimensional and catenary is that $B$ is equidimensional and universally catenary (see \cite[Theorem 3.8]{RATLIFF_II}, \cite[Lemma 2.2]{SUV_MULT}).
	
	\noindent
	To circumvent any assumption on $G$, we work directly over the algebra $B$, and our primary tool is the technical result of \autoref{lem_properties_closed_point}.
\end{remark}	

The next corollary slightly improves some criteria from \cite{SUV_MULT} by weakening an assumption of universally catenary to just catenary.
In the $\NN$-graded case (i.e., $p=1$), to recover the original definition of Simis, Ulrich and Vasconcelos \cite{SUV_MULT}, we set $e_t(A, B) := e_t\left(\dim(B)-1; A, B\right)$ (notice that one can have the strict inequality $\dim(X) < \dim(B) -1$).

\begin{corollary}[{cf., \cite[Corollary 2.7, Theorem 3.3]{SUV_MULT}}]
	\label{cor_applications_NN}
	Adopt the notation of \autoref{thm_criteria} in an $\NN$-graded setting {\rm(}i.e., $p=1${\rm)}. 
	Let $G = \gr_{A_+B}(B)$ and $t \ge 1$.
	Consider the following conditions: 
	\begin{enumerate}[\rm (a)]
		\item $B$ is integral over $A$ and $B_\pp = \sum_{i=0}^{t-1}[B]_iA_\pp$ for every minimal prime $\pp$ of $A$.
		\item $\dim\left([B]_tG\right) < \dim(G) =\dim(B)$.
		\item $e_t(A, B) = 0$.
	\end{enumerate}
	In general, the implications {\rm(a)} $\Rightarrow$ {\rm(b)} $\Rightarrow$ {\rm(c)} hold. 
	If $B$ is equidimensional and catenary, the three conditions are equivalent.
\end{corollary}
\begin{proof}
	The implications (a) $\Rightarrow$ (b) $\Rightarrow$ (c) follow from \cite[Corollary 2.4, Proposition 3.2]{SUV_MULT}.
	Thus we only need to show (c) $\Rightarrow$ (a) under the assumption that $B$ is equidimensional and catenary.
	Let $\overline{B} = B/(0:_BB_+^\infty)$ and $\overline{A} = A/A\cap(0:_BB_+^\infty)$.
	Notice that $\overline{B}$ is equidimensional and catenary, and that $A \hookrightarrow B$ is integral if and only if $\overline{A} \hookrightarrow \overline{B}$ is.
	As $B$ is equidimensional, we have either $\dim(\overline{B})=\dim(B)$ or $\overline{B}=0$.
	Therefore, we may assume that $r=\dim(X) = \dim(\overline{B}) -1 = \dim(B) - 1$.
	Since $e_\infty(r;A,B)=e_\infty(A, B) \le e_t(A,B) = 0$, \autoref{thm_criteria}(i) implies that $A \hookrightarrow B$ is integral.
	
	Next, we consider the finitely generated graded $A$-modules $M = \sum_{i=0}^{t-1}[B]_iA \subseteq B$ and $B$.
	The associativity formula of \autoref{lem_assoc_j_mult}  yields the equalities 
	$$
	j_r(M) = \sum_{\pp} \length_{A_\pp}\left(M_\pp\right) e(A/\pp) \quad \text{ and } \quad j_r(B) = \sum_{\pp} \length_{A_\pp}\left(B_\pp\right) e(A/\pp),
	$$
	where, in both summations,  $\pp \in \Spec(A)$ runs through the minimal primes of $A$ such that $\dim(A/\pp)=r+1$ and $\pp \cap R = \mm$.
	Notice that $A$ is also equidimensional.
	By additivity, we have $j_r(B) - j_r(M) = j_r(B/M) =e_t(A, B) =  0$ (see \autoref{lem_j_mult_add}).
	Since $\length_R([B]_1/[A]_1) < \infty$, we obtain $M_\pp = B_\pp$ for all $\pp \in \Spec(A)$ with $\pp \cap R \subsetneq \mm$.
	Therefore, it follows that $M_\pp = B_\pp$ for every minimal prime $\pp$ of $A$.
\end{proof}
	
	Before proving the above theorem, we need a couple of technical lemmas.
	The first one is inspired by Rees' original proof of his celebrated integral dependence criterion \cite{REES_CRITERION, REES_AMAO_THM}.
	
	\begin{lemma}
		\label{lem_properties_closed_point}
		Let $T$ be a standard graded $R$-algebra, and $\pi : X = \Proj(T) \rightarrow \Spec(R)$ be the natural projective morphism.
		Let $x \in X$ be a closed point and $P \subset T$ be the corresponding relevant homogeneous prime ideal. 
		Then the following statements hold:
		\begin{enumerate}[\rm (i)]
			\item $\pi$ maps $x$ to the closed point of $\Spec(R)$, and thus $P \supset \mm T$.
			\item ${\rm length}_R\big(\HH^0\big(X, \OO_X/\,\widetilde{P}\big)\big) = [\kappa(x) : \kappa]$, where $\kappa(x) = \OO_x/\mm_x$ is the residue field at $x$.
			\item $\lambda_x(n) := {\rm length}_R\big(\HH^0\big(X, \OO_X/\,\widetilde{P^n}\big)\big)$ eventually coincides with a polynomial $P_x(n)$ of degree $d_x := \dim(\OO_x)$ and whose leading coefficient is $[\kappa(x) : \kappa] \cdot e(\OO_x) / d_x!$.
			\item Let $H = [P]_d$ such that $(H) = [P]_{\ge d} = \bigoplus_{v \ge d} [P]_v$ {\rm(}i.e., $d$ is larger than or equal to the degrees of a set of homogeneous generators of $P${\rm)}.
			Then the function $L_d(v, n) := {\rm length}_R \left([T]_{v+nd} \big/ H^n[T]_v\right)$ eventually becomes a polynomial $P_d(v, n)$.
			Furthermore, we actually have $P_d(v, n) = P_x(n)$.
		\end{enumerate}
	\end{lemma}
	\begin{proof}	
		(i) The morphism $\pi$ is closed, hence it should send the closed point $x \in X$ to the unique closed point in $\Spec(R)$.
		It clearly follows that $P \supset \mm T$.
		
		(ii) Since $\HH^0\big(X, \OO_X/\,\widetilde{P}\big) \cong \kappa(x)$, the equality is clear.
		
		(iii) For all $n \ge 1$, we have $\HH^0\big(X, \OO_X/\,\widetilde{P^n}\big) \cong \OO_x / \mm_x^n$, and so we obtain 
		$$
		\lambda_x(n) \;=\; [\kappa(x):\kappa] \cdot {\rm length}_{\OO_x}(\OO_x/\mm_x^n).
		$$
		This settles the claim, since ${\rm length}_{\OO_x}(\OO_x/\mm_x^n)$ is eventually the Hilbert-Samuel polynomial of the local ring $(\OO_x, \mm_x)$.
		
		(iv) The fact that $L_d(v, n)$ eventually coincides with a polynomial $P_d(v,n)$ follows from the existence of mixed Buchsbaum-Rim multiplicities (see \autoref{thm_poly_KT} or  \cite{KLEIMAN_THORUP_MIXED}). 
		We fix $n_0$ big enough so that $L_d(v, n_0)$ eventually becomes the polynomial $P_d(v, n_0)$ in $v$.
		However, for $v \gg 0$, we have 
		$$
		L_d(v,n_0) = {\rm length}_R \left({\left[T/P^{n_0}\right]}_{v+n_0d}\right) = {\rm length}_R\left(\HH^0\left(X, \OO_X/\,\widetilde{P^{n_0}}\,(v+n_0d)\right)\right) = \lambda_x(n_0),
		$$
		and the last equality follows because $x \in X$ is a closed point.
		This shows that $P_d(v, n) = P_x(n)$.
	\end{proof}

\begin{lemma}
	\label{lem_dim_closed_point}
	With the same notation of \autoref{thm_criteria}, assume that $B$ is equidimensional and catenary.
	Then, for any closed point $x \in X$, we have $d_x=\dim(\OO_x) = r$.
\end{lemma}	
\begin{proof}
	We may assume $(0:_BB_{++}) = 0$, and so $\dim(B) = r+ p$.
	Notice that $B$ is naturally a *local ring with *maximal ideal given by $\MM = \mm + B_+$ and that $r+p=\dim(B) = \HT(\MM)$; see \cite[\S 1.5]{BRUNS_HERZOG}.
	Since $B$ is equidimensional and catenary, any maximal chain of homogeneous prime ideals has length $r+p$.	
	Let $x \in X$ be a closed point with associated prime $P \subset B$.
	As $x \in X$ is a closed point,  $\dim(B/P)=\dim\big(\overline{\{x\}}\big)+p=p$. 
	Therefore, we obtain $d_x = \dim(\OO_x) = \dim(B_P) = \dim(B) - \dim(B/P) = (r+p) - p = r$.
\end{proof}
	
We are now ready to prove the main result of this section.
	
	\begin{proof}[Proof of \autoref{thm_criteria}]
		We may substitute $A$ and $B$ by $\overline{A}=A/A\cap(0:_BB_{++}^\infty)$ and $\overline{B}=B/(0:_BB_{++}^\infty)$, respectively.
		By utilizing the Segre embedding, we obtain the isomorphisms $X \cong \Proj(T)$ and $Y \cong \Proj(S)$, where $S := \bigoplus_{n \ge 0} [A]_{(n,\ldots,n)}$ and $T := \bigoplus_{n \ge 0} [B]_{(n,\ldots,n)}$.
		Then we see $f : U \subseteq X=\Proj(T) \rightarrow Y=\Proj(S)$ as the morphism associated to the inclusion $S \subseteq T$ of standard graded $R$-algebras.
		
		By considering the substituted polynomial $P_{\mathbf{1}}^{A,B}(n,\ldots,n)$ (see \autoref{thm_relative_mixed_mult}), we can compute the equality 
		$$
		e(S, T) \;=\; \bigoplus_{\beta \in \NN^p,\, |\beta| = r} \frac{r!}{\beta_1!\cdots\beta_p!} \, e(\beta; A, B).
		$$
		Similarly, as a consequence of  \autoref{lem_properties}(iii), by choosing $t > 0$ large enough and considering the substituted polynomial $P_{(t,\ldots,t)}^{A, B}(n,\ldots,n)$, we can compute the equality
		$$
		e_\infty(S, T) \;=\; \bigoplus_{\beta \in \NN^p,\, |\beta| = r} \frac{r!}{\beta_1!\cdots\beta_p!} \, e_\infty(\beta; A, B).
		$$
		From these reductions, we may instead prove the theorem in terms of the vanishing of $e(S, T)$ and $e_\infty(S, T)$.
		
		(i)
		Suppose that $S \subseteq T$ is an integral extension.
		Then we may choose $t > 0$ large enough such that $[T]_n = [S]_{n-t+1}[T]_{t-1}$ for all $n \ge 0$.
		This shows that $e_\infty(S, T) = 0$.
		
		Conversely, suppose that $S \subseteq T$ is not an integral extension and that $B$ is equidimensional and catenary.
		Then $S_+T$ is a relevant ideal of $T$ (i.e., $\sqrt{S_+T} \not\supseteq T_+$).
		Let $x \in X$ be a closed point with associated prime $P \subset T$ such that $P \supseteq S_+T$.
		From \autoref{lem_dim_closed_point}, we get the equality $d_x = \dim(\OO_x) = r$.
		By \autoref{lem_properties_closed_point}, for large enough integers $d > 0$ and $v > 0$,  we obtain 
		$$
		\length_R\big([T]_{v+nd}/[S]_{nd}[T]_v\big) \;\ge\; \length_R\big([T]_{v+nd}/H^n[T]_v\big) \;=\; P_x(n) \quad \text{ for } \quad n \gg 0,
		$$
		where $H = [P]_d$ and $P_x(n)$ is a polynomial of degree $r$ with leading coefficient $[\kappa(x) : \kappa] \cdot e(\OO_x) / r!$.
		It then follows that 
		$$
		e_\infty(S, T) \;\;=\;\; \frac{r!}{d^r} \lim_{n \to \infty}  \frac{\length_R\big([T]_{v+nd}/[S]_{nd}[T]_v\big)}{n^r} \;\;\ge\;\; \frac{r!}{d^r} \lim_{n \to \infty} \frac{P_x(n)}{n^r} \;\;=\;\; \frac{[\kappa(x):\kappa] \cdot e(\OO_x)}{d^r} \;\;>\;\; 0.
		$$
		Therefore, $e_\infty(S, T) \ge 1$ is a positive integer, as required.
		
		(ii) After proving part (i), the result of \autoref{cor_applications_NN} is already valid.
		In particular, the statement about birationality follows.
	\end{proof}

	It should be mentioned that in a multigraded setting the inclusion $A \subseteq B$ may not be integral and yet we can still obtain a finite morphism:
	
	\begin{example}
		Let $\kk$ be a field and consider the following inclusion of standard $\NN^2$-graded $\kk$-algebras
		$$
		A = \kk[x_1,x_2,y_1,y_2] \;\; \hookrightarrow \;\; B = \frac{\kk[x_1,x_2,x_3,y_1,y_2,y_3]}{\big( x_3(y_1,y_2,y_3), \, y_3(x_1,x_2,x_3)\big)}
		$$
		with $\deg(x_i) = \ee_1 \in \NN^2$ and $\deg(y_i) = \ee_2 \in \NN^2$.
		Then $f : \biProj(B) \subset \PP_\kk^2 \times_\kk \PP_\kk^2 \rightarrow \biProj(A) = \PP_\kk^1 \times_\kk \PP_\kk^1$ is a finite morphism (in fact, it is an isomorphism since $B/(0:_BB_{++}) \cong B/(x_3, y_3)$).
		However, the extension $A \hookrightarrow B$ is not integral.
	\end{example} 
	
	Finally, we briefly discuss the particular case of graphs of rational maps.
	
	\begin{remark}
		Let $\kk$ be a field and $\mathfrak{d}_1 \subseteq \mathfrak{d_2} \subseteq \HH^0\big(\PP_\kk^r, \OO_{\PP_\kk^r}(d)\big)$ be two linear systems. 
		Let $\FF_1 : \PP_\kk^r \dashrightarrow \PP_\kk^{s_1}$ and $\FF_2 : \PP_\kk^r \dashrightarrow \PP_\kk^{s_2}$ be the respective rational maps.
		Let $\Gamma_i$ and $E_i$ be the graph (blow-up of the base locus) and exceptional divisor of $\FF_i$, respectively.
		Then the following conditions are equivalent: 
		\begin{enumerate}[\rm (a)]
			\item The induced rational map $f : \Gamma_2 \subset \PP_\kk^r \times_\kk \PP_\kk^{s_2} \dashrightarrow \Gamma_1 \subset \PP_\kk^r \times_\kk \PP_\kk^{s_1}$ between graphs is a finite birational morphism.
			\item The projective degrees of $\FF_1$ and $\FF_2$ coincide (see \cite[Example 19.4]{HARRIS}, \cite[\S 7.1.3]{DOLGACHEV}).
			\item The multidegrees of $E_1$ and $E_2$ coincide.
		\end{enumerate}
		The equivalence between (a) and (b) follows from \autoref{thm_criteria}.
		Since $\OO_{\Gamma_i}(-E_i) \cong \OO_{\Gamma_i}(-d, 1)$, we have the short exact sequence $0 \rightarrow \OO_{\Gamma_i}(-d,1) \rightarrow \OO_{\Gamma_i} \rightarrow \OO_{E_i} \rightarrow 0$, and so we can compute the following equality on multidegrees (see, e.g., \cite{POSITIVITY}):
		$$
		\deg^{(n_1,n_2)}(E_i) \;=\; d\cdot\deg^{(n_1+1,n_2)}(\Gamma_i) - \deg^{(n_1,n_2+1)}(\Gamma_i) \quad \text{ for all } \quad n_1+n_2=r-1.
		$$
		Therefore, since $\deg^{(r,0)}(\Gamma_i) = 1$, it follows that parts (b) and (c) are equivalent.
		
		Alternatively, the equivalence between parts (a) and (c) could be deduced as a very particular case of the main result of \cite{PTUV_MULT}.
		Indeed, important recent work of Polini, Trung, Ulrich and Validashti \cite{PTUV_MULT} shows that, over an equidimensional and universally catenary Noetherian local ring, two ideals $I \subset J$ have the same integral closure if and only if their multiplicity sequences coincide.
	\end{remark}

%

	\bibliography{references}

\end{document}